\documentclass[12pt,reqno]{amsart}

\usepackage[utf8]{inputenc}
\usepackage[T1]{fontenc}
\usepackage{lmodern}  

\usepackage{amsmath, amsfonts, amssymb, amsthm, amscd, amsbsy}
\usepackage{mathtools}  
\usepackage{mathrsfs}   

\usepackage{enumitem}   
\usepackage{graphicx}   
\usepackage{lineno}     
\usepackage[numbers,sort&compress]{natbib}  

\usepackage[a4paper]{geometry}
\numberwithin{equation}{section}

\theoremstyle{plain}
\newtheorem{theorem}{Theorem}[section]

\newtheorem{corollary}{Corollary}[section]

\newtheorem{lemma}{Lemma}[section]

\allowdisplaybreaks

\newtheorem{theorema}{Theorem}[section]

\usepackage[colorlinks,breaklinks,
linkcolor=blue,
citecolor=green,
urlcolor=NavyBlue]{hyperref}

\title[Heisenberg Uncertainty Principle on Orthants]{Heisenberg Uncertainty Principle on half spaces and Orthants: Best constants, Optimizers and Stability}

\author[N.~Lam]{Nguyen Lam}
\address{School of Science and the Environment, Grenfell Campus, Memorial University of Newfoundland, Corner Brook, NL A2H 5G4, Canada}
\email{nlam@mun.ca}

\author[Y.~Lodha]{Yukta Lodha}
\address{Department of Mathematics, University of Connecticut, Storrs, CT 06269, USA}
\email{yukta.lodha@uconn.edu}

\author[G.~Lu]{Guozhen Lu}
\address{Department of Mathematics, University of Connecticut, Storrs, CT 06269, USA}
\email{guozhen.lu@uconn.edu}

\author[A.~N.~Sengupta]{Ambar N.~Sengupta}
\address{Department of Mathematics, University of Connecticut, Storrs, CT 06269, USA}
\email{ambar.sengupta@uconn.edu}

\date{\today}
\thanks{}


\newcommand{\mbr}{{\mathbb R}}
\newcommand{\R}{{\mathbb R}}
\newcommand{\Rnkp}{{\mathbb R}^n_{k,+}}
\newcommand{\RnAp}{{\mathbb R}^n_{A,+}}

\newcommand{\tlv}{{\tilde v}}

\newcommand{\supp}{{\rm supp}}
\newcommand{\pxi}{{\Big( \prod_{i=n-k+1}^{n} x_i}\Big)}

\begin{document}

\begin{abstract}
Though the sharp Heisenberg Uncertainty Principle has been extensively studied in the entire Euclidean spaces, the counterpart on the half spaces or more general orthants has been missing in the literature. 
We investigate the sharp Heisenberg Uncertainty Principle on orthants by computing explicitly the optimal constant and determining all possible extremal functions. Moreover, we establish several stability estimates of the Heisenberg Uncertainty Principle on the half spaces and orthants.
\end{abstract}

\maketitle

\section{Introduction}

Let $\Omega$ be a domain in $\R^n$ containing the origin. Then the celebrated Hardy inequality states that for all $u \in C_c^\infty(\Omega)$, one has 
\begin{equation} \label{Hardy}
    \int_{\Omega} |\nabla u(x)|^2 \, dx \geq H(\Omega)  \int_{\Omega} \frac{|u(x)|^2}{|x|^2} \, dx.
\end{equation}
Here, the optimal constant $H(\Omega)$ in \eqref{Hardy} is $\frac{(n-2)^2}{4}$ even though it cannot be achieved by nontrivial functions. This fact is now very well-understood. Actually, the Hardy type inequalities, along with their various improvements and applications, have been extensively investigated in the literature, which is remarkably rich and vast. The interested reader is referred to
\cite{BFT04, BM98, BV, CKLL24, CZ13, CGMSO18, DD14, DDLL25, DoFLL23, DoLL24, DLLZ, DLL22, DPH25, Flynn20, FLL22, FLL21, FLL25, FLLM23, FLY24, FS08, GR23, 
GL2017, GLMP, GLMW, HT25, HY22, LL23, LLZ19, LLZ20, LuYang2, LuYang1, LuYang3, NVH20, NVH19,  Wang22, YSK15}, to name just a few. We also mention the monographs \cite{BEL, GM1, KMP2007, KP, OK}, which serve as standard references in this field.

The situation is very different  for domains having $0$ on their boundary. Indeed, in this case, it was showed in \cite{Caz10, Fall12, FM12}, for instance, that the sharp constants in the Hardy inequality \eqref{Hardy} can be anywhere between $\frac{(n-2)^2}{4}$ and $\frac{n^2}{4}$. Moreover, it was also pointed out that the optimal constant can be attained by nontrivial functions as long as it is strictly less than $\frac{n^2}{4}$. In particular, in the case where $\Omega$ is a half-space $\R^{n-1}\times \R_{>0}$, the following Hardy inequality has been studied (see \cite{FTT09}):
\begin{equation} \label{Hardy_halfspace}
    \int_{\R^{n-1}\times \R_{>0}} |\nabla u(x)|^2 \, dx \geq \frac{n^2}{4}  \int_{\R^{n-1}\times \R_{>0}} \frac{|u(x)|^2}{|x|^2} \, dx.
\end{equation}
The constant $\frac{n^2}{4}$ is sharp in \eqref{Hardy_halfspace} and it cannot be achieved by nontrivial functions. Here $\R_{>0}=(0,\infty)$.

The Hardy inequality has also been investigated in the setting of orthants.  By an {\em orthant} we mean subset of $\R^n$ of the form $\R_{k,+}^n:=\R^{n-k}\times \R_{>0}^k$, 
where $k\in\{1,\ldots, n\}$. Indeed, in this case, by lifting the orthant $\R_{k,+}^n$ to the whole Euclidean space $\mathbb{R}^{n+2k}$, Su and Yang showed in \cite{SuYang2012} that 

\begin{equation} \label{Hardy_orthant}
    \int_{\R_{k,+}^n} |\nabla u(x)|^2 \, dx \geq \frac{(n+2k-2)^2}{4}  \int_{\R_{k,+}^n} \frac{|u(x)|^2}{|x|^2} \, dx.
\end{equation}
Also, the constant $\frac{(n+2k-2)^2}{4}$ is sharp.

As a simple consequence of the Hardy inequality \eqref{Hardy} and the H\"{o}lder inequality, one obtains
\begin{align}
     \int_{\mathbb{R}^n} |\nabla u(x)|^2 \, dx  \; \int_{\mathbb{R}^n} |x|^2 |u(x)|^2 \,dx
&\geq \frac{(n-2)^2}{4} \int_{\mathbb{R}^n} \frac{|u(x)|^2}{|x|^2} \, dx \; \int_{\mathbb{R}^n} |x|^2 |u(x)|^2 dx \nonumber\\
&\geq \frac{(n -2)^2}{4} \left( \int_{\mathbb{R}^n} |u(x)|^2 \, dx\right)^2. \label{wHUP}
\end{align}
The inequality \eqref{wHUP} is a mathematical formulation of the well-known Heisenberg Uncertainty Principle (HUP) \cite{Heis1927}, which expresses a fundamental limitation on the simultaneous localization of a function and its Fourier transform. In its classical form, the HUP asserts that a nontrivial function \( u \in L^2(\mathbb{R}^n) \) cannot be sharply localized both in position and in momentum spaces, and quantitatively this is represented by  
\[
\left( \int_{\mathbb{R}^n} |\nabla u(x)|^2\,dx \right)\left( \int_{\mathbb{R}^n} |x|^2 |u(x)|^2\,dx \right)
 \gtrsim \left( \int_{\mathbb{R}^n} |u(x)|^2\,dx \right)^2.
\]  
The physics version of the HUP involves the variance, not the expectation of squares of observables as above, but one can be deduced from the other readily. The HUP originates in Heisenberg's work \cite{Heis1927}, and in more precise form in Kennard \cite{Kennard1927} (see  
\cite{Fol97} for a discussion in the context of harmonic analysis).  This inequality shows that any attempt to reduce the spatial spread of \(u\) necessarily increases the spread of its gradient (or frequency), reflecting the fundamental trade-off between position and momentum. It should be noted that the constant \(\frac{(n-2)^2}{4}\) in \eqref{wHUP} is not sharp. Surprisingly, the sharp form of the HUP can be easily derived by the following elementary method. 
Indeed, we have by the H\"{o}lder inequality and the divergence theorem (assumed applicable) that

\begin{align}
    \int_{\mathbb{R}^n} |\nabla u(x)|^2 \, dx  \; \int_{\mathbb{R}^n} |x|^2 |u(x)|^2 \,dx &\geq \left(\int_{\mathbb{R}^n} \left(\nabla u(x)\right) \cdot xu(x) \,dx\right)^2 \nonumber\\ 
    &= \left(\frac{1}{2}\int_{\mathbb{R}^n} \left(\nabla u(x)^2\right)\cdot x \,dx\right)^2 \nonumber\\
     &= \left(-\frac{1}{2}\int_{\mathbb{R}^n} |u(x)|^2 \text{div}(x) \,dx\right)^2 \nonumber\\
     &=\frac{n^2}{4} \left( \int_{\mathbb{R}^n} |u(x)|^2 \, dx\right)^2. \label{HUP}
\end{align}
In fact, it has been known that the constant \(\frac{n^2}{4}\) is optimal in the Heisenberg Uncertainty Principle, and equality is attained precisely by Gaussian functions $\alpha e^{-\beta\left\vert x\right\vert ^{2}}$ with $\beta>0$, which play a central role in quantum mechanics and harmonic analysis. See, for instance, \cite{CFLL24}.

Unlike the Hardy inequality, a formulation of the HUP on half-spaces or orthants has not yet appeared in the literature. Clearly, as in the inequality \eqref{wHUP}, combining the Hardy inequality on orthants \eqref{Hardy_orthant} and the H\"{o}lder inequality yields the following HUP on orthants:

\begin{align}
     \int_{{\Rnkp}} |\nabla u(x)|^2 \, dx  \; \int_{{\Rnkp}} |x|^2 |u(x)|^2 \,dx
&\geq \frac{(n+2k-2)^2}{4} \int_{{\Rnkp}} \frac{|u(x)|^2}{|x|^2} \, dx \; \int_{{\Rnkp}} |x|^2 |u(x)|^2 dx \nonumber\\
&\geq \frac{(n+2k -2)^2}{4} \left( \int_{{\Rnkp}} |u(x)|^2 \, dx\right)^2. \label{wHUPo}
\end{align}

However, by \eqref{HUP}, it is very unlikely that the constant $\frac{(n+2k -2)^2}{4}$ is optimal in \eqref{wHUPo}. Our first main purpose of this paper is to establish and study inequalities of the type of Heisenberg's uncertainty principle in the setting of orthants. Our goal is to find the optimal constant in \eqref{wHUPo} and determine all functions $u$ for which equality holds. 

Our first principal result can be read as follows. Let $S({\Rnkp})$ be the completion of $C_c^\infty({\Rnkp})$ under the norm $\left (\int_{{\Rnkp}} |\nabla u(x)|^2 \, dx \right )^{\frac{1}{2}} +  \left (\int_{{\Rnkp}} |x|^2 |u(x)|^2 dx\right )^{\frac{1}{2}}$; then:

\begin{theorem}[{\bf HUP on orthants}]\label{T:HeisOrth} 
Let $n \geq 1$, and $k\in\{1,\ldots, n\}$. Then 
\begin{equation}\label{E:HUPorthant}
\int_{{\Rnkp}} |\nabla u(x)|^2 \, dx  \; \int_{{\Rnkp}} |x|^2 |u(x)|^2 dx
\geq \frac{(n + 2k)^2}{4} \left( \int_{{\Rnkp}} |u(x)|^2 \, dx\right)^2  
\end{equation}
for all $u \in S({\Rnkp})$.
The constant $\frac{(n+2k)^2}{4}$ in (\ref{E:HUPorthant}) is sharp, and equality holds in (\ref{E:HUPorthant}) if and only if $u(x)$ is of the form $\Big( \prod_{i=n-k+1}^{n}x_i \Big)\alpha e^{-\beta|x|^2}$, for some constants $\alpha$ and $\beta$, with $\beta>0$.
\end{theorem}

We note here that by our Theorem \ref{T:HeisOrth}, the optimal constant of the HUP on orthants can be improved from $\frac{n^2}{4}$, the sharp constant of the HUP on $\mathbb{R}^n$, to $\frac{(n + 2k)^2}{4}$. Also, as a consequence, we obtain the following HUP on half-spaces, which seems to be the first result about the HUP on domains having $0$ on their boundary:

\begin{theorem}[{\bf HUP on half spaces}]\label{T1.1} 
There holds
\begin{equation}\label{E:HUPhalfspace}  
\int_{\R^{n-1}\times \R_{>0}} |\nabla u(x)|^2 \, dx  \; \int_{\R^{n-1}\times \R_{>0}} |x|^2 |u(x)|^2 dx
\geq \frac{(n + 2)^2}{4} \left( \int_{\R^{n-1}\times \R_{>0}} |u(x)|^2 \, dx\right)^2  
\end{equation}
for all $u \in S(\R^{n-1}\times \R_{>0})$.
The constant $\frac{(n+2)^2}{4}$ in (\ref{E:HUPhalfspace}) is sharp, and equality holds in (\ref{E:HUPhalfspace}) if and only if $u$ is of the form $\alpha x_ne^{-\beta|x|^2}$, for some constants $\alpha$ and $\beta$, with $\beta>0$.
\end{theorem}

A natural question once we can determine the sharp constants and all the optimizers of \eqref{E:HUPorthant} and \eqref{E:HUPhalfspace} is to study their stability. This type of question was first raised by Brezis and Lieb in \cite{BL85}. More clearly, since the optimal constant and the full family of extremal functions for the celebrated Sobolev inequality are known explicitly, Brezis and Lieb asked in \cite{BL85} whether the Sobolev inequality can be improved by controlling the difference between the two sides of the inequality in terms of the distance of the function to the set of extremal functions. This question, known as the stability problem for the Sobolev inequality, was answered affirmatively by Bianchi and Egnell who proved in \cite{BE91} the following stability estimate: there exists a constant $c_{BE} > 0$ such that
\begin{equation}\label{StabilityS}
\int_{\mathbb{R}^{n}}\left\vert \nabla u(x)\right\vert ^{2}dx-S_{n}\left(
\int_{\mathbb{R}^{n}}|u(x)|^{\frac{2n}{n-2}}dx\right)  ^{\frac{n-2}{n}}\geq
c_{BE}\inf_{U\in E_{Sob}}\int_{\mathbb{R}^{n}}\left\vert \nabla\left(
u(x)-U(x)\right)  \right\vert ^{2}dx
\end{equation}
Here $S_{n}$ is the sharp Sobolev
constant and $E_{Sob}$ is the manifold of the optimizers of the Sobolev
inequality. In other words, the Sobolev deficit controls the squared distance in the gradient norm to the manifold of optimizers. This is optimal in terms of the powers involved and the metric used. See \cite{FN19, FZ22}. The Brezis-Lieb question and the Bianchi-Egnell answer have initiated a deep line of research on the quantitative stability of functional and geometric inequalities measuring how close a function is to the optimizers in terms of their deficit. The literature on this topic is extensive. Therefore, we refer the interested reader to \cite{BWW03, BDNN20, CF13, CLT23, CFW13, CFMP09, FJ1, FJ2, FN19, FZ22, LW99}, for instance, for detailed studies on the stability of Sobolev-type inequalities.  

It is worth noting that the stability constants and the attainability of the corresponding stability inequalities have generally been unexplored in the literature. In particular, precise information on the stability constant $c_{BE}$ was mostly unavailable until very recently. In a recent paper \cite{DEFFL}, Dolbeault, Esteban, Figalli, Frank, and Loss pioneered the rigorous study of the constant $c_{BE}$ in \eqref{StabilityS}. They provided optimal lower bounds for $c_{BE}$ in the asymptotic regime as the dimension $n \uparrow \infty$, and established stability results for the Gross Gaussian log-Sobolev inequality \cite{Gro75} as an application. One of the major contributions in \cite{DEFFL} is to develop a gradient flow method to pass from the local stability to the global stability of the Sobolev inequality.    More recently, Chen, Tang, and the third author studied explicit lower bounds for stability in Hardy-Littlewood-Sobolev inequalities, deducing also explicit lower bounds for higher and fractional order Sobolev inequalities in \cite{CLT24}. They additionally obtained optimal asymptotic lower bounds in \cite{CLT242, CLT243} for Hardy-Littlewood-Sobolev and higher fractional Sobolev inequalities as 
$n \uparrow \infty$ for $0<s<\frac{n}{2}$. The latter allowed them to derive global stability for the log-Sobolev inequality on the sphere, originally established by Beckner \cite{Beckner93}, and to sharpen earlier local stability results from \cite{CLT23}.
More recently, the authors Chen, et al established in \cite{CLTW} the optimal stability of the Sobolev inequality on the Heisenberg group.
Due to the failure of the P\'olya-Szeg\"{o} inequality and the Riesz rearrangement inequality in the Heisenberg group,  both the gradient flow or integral flow strategy in \cite{DEFFL, CLT24, CLT242, CLT243} no longer work on the Heisenberg group. Therefore, the authors of \cite{CLTW} developed a new method of employing the CR Yamabe flow to pass from the local stability to the global stability.

In their investigation of the stability of the Heisenberg Uncertainty Principle (HUP), McCurdy and Venkatraman employed the concentration-compactness techniques and proved in \cite{MV21} the existence of universal constants $C_{1} > 0$ and $C_{2}(n) > 0$ such that  
$$
\delta_{2}(u) \geq C_{1}\left(\int_{\mathbb{R}^{n}} |u(x)|^{2}\,dx\right) d_{1}^{2}(u, E_{HUP}) + C_{2}(n)d_{1}^{4}(u, E_{HUP}).
$$
Here, the HUP deficit is defined by  
$$
\delta_{2}(u) := \left(\int_{\mathbb{R}^{n}} |\nabla u(x)|^{2} \, dx\right) \left(\int_{\mathbb{R}^{n}} |x|^{2}|u(x)|^{2} \, dx\right) - \frac{n^{2}}{4}\left(\int_{\mathbb{R}^{n}} |u(x)|^{2} \, dx\right)^{2},$$ 
and
$$E_{HUP}:=\left\{  \alpha e^{-\beta\left\vert x\right\vert ^{2}}:\alpha
\in\mathbb{R}\text{, }\beta>0\right\}  $$ 
is the set of the optimizers of the HUP
\eqref{HUP}, and the distance from $u$ to a set $A$ is given by  
$$
d_{1}(u, A) := \inf_{v \in A} \|u - v\|_{2}.
$$
Consequently, a small deficit $\delta_{2}(u) \approx 0$ implies that $u$ is close in $L^{2}(\mathbb{R}^{n})$ to a Gaussian of the form $u \approx \alpha e^{-\beta |x|^{2}}$ for some $\alpha \in \mathbb{R}$, $\beta > 0$.

A simpler, constructive proof was later provided by Fathi in \cite{F21}, yielding explicit values $C_{1} = \frac{1}{4}$ and $C_{2} = \frac{1}{16}$, although these constants are not sharp.  
Subsequently, the authors in \cite{CFLL24} derived the following sharp stability result for the HUP.

\begin{theorema}
\label{A}
There holds
$$
\delta_{1}(u) := \left(\int_{\mathbb{R}^{n}} |\nabla u(x)|^{2} \, dx\right)^{\frac{1}{2}} \left(\int_{\mathbb{R}^{n}} |x|^{2}|u(x)|^{2} \, dx\right)^{\frac{1}{2}} - \frac{n}{2}\int_{\mathbb{R}^{n}} |u(x)|^{2} \, dx \geq d_{1}^{2}(u, E_{HUP}).
$$
Moreover, the inequality is sharp, and equality can be attained by nontrivial functions $u \notin E_{HUP}$.
\end{theorema}

A key ingredient in \cite{CFLL24} is the establishment of scale non-invariant HUP identity together with the  following Gaussian-type Poincaré inequality: for all \(\lambda \neq 0\),
\[
\int_{\mathbb{R}^{n}} |\nabla u(x)|^{2} e^{-\frac{1}{2|\lambda|^{2}}|x|^{2}} \, dx \geq \frac{1}{|\lambda|^{2}} \inf_{c} \int_{\mathbb{R}^{n}} |u(x) - c|^{2} e^{-\frac{1}{2|\lambda|^{2}}|x|^{2}} \, dx.
\]

When \(\lambda = 1\), this reduces to the classical Gaussian Poincaré inequality, a fundamental tool in the analysis of Gaussian measures. A comprehensive reference on this topic is the monograph by Bakry, Gentil, and Ledoux \cite{BGL14}. We also refer the reader to the seminal work of Leonard Gross \cite{Gro75} from the 1970s, which laid the foundations for the study of logarithmic Sobolev and Gaussian Poincaré inequalities.

As a corollary, Theorem \ref{A} implies that
$$
\delta_{2}(u) \geq n\left(\int_{\mathbb{R}^{n}} |u(x)|^{2} \, dx\right)d_{1}^{2}(u, E_{HUP}) + d_{1}^{4}(u, E_{HUP}),
$$
where equality again holds for some nontrivial functions $u \notin E_{HUP}$. 

In \cite{LLR25}, the stability of the HUP has been studied further. More clearly, the authors in \cite{LLR25} established the following stability result for the stability of the HUP (Theorem \ref{A}):

\begin{theorema}
\label{B}
There holds
$$
\delta_{1}\left(  u\right)  -d_{1}^{2}(u,E_{HUP})\geq d_{1}^{2}(u,F).
$$
\end{theorema}
Here \[
F:=\left\{  \left(  \alpha+\mathbf{\gamma}\cdot x\right)
e^{-\beta\left\vert x\right\vert ^{2}}:\alpha\in\mathbb{R}\text{,
}\mathbf{\gamma}\in\mathbb{R}^{n}\text{, }\beta>0\right\}.
\]

Motivated by the stability results in \cite{CFLL24, LLR25} and our HUP on orthants (Theorem \ref{T:HeisOrth}), our next primary aim is to investigate a stability result for the inequality (\ref{E:HUPorthant}). In other words, we would like to prove a sharp lower bound for the deficit in the inequality (\ref{E:HUPorthant}) in terms of the distance of $u$ from the set of optimizers. More explicitly, define the Heisenberg deficit on an orthant by
\begin{equation}\label{HeisenbergDeficit}
    \rho_{1}(u)
    = 
    \left( \int_{{\Rnkp}} |\nabla u(x)|^{2}\,dx \right)^{\frac{1}{2}}
    \left( \int_{{\Rnkp}} |x|^{2} |u(x)|^{2}\,dx \right)^{\frac{1}{2}}
    - \frac{n+2k}{2} \int_{{\Rnkp}} |u(x)|^{2}\,dx
\end{equation}
and let $\tilde{E}=\left\{\left(\prod_{i=n-k+1}^{n} x_i\right)\alpha e^{-\beta |x|^2}:  \alpha \in \mbr, \beta>0\right\}$ be the set of all the optimizers for the HUP on orthants \eqref{E:HUPorthant}.
Then, among many other results, we have:
\begin{theorem}\label{T2}
There holds
\[
\rho_1 \left( u \right) \;\; \ge \;\; \inf_{v \in\tilde{E}} \| u - v \|_{2}^2 .
\]
 The equality holds 
if and only if $u$ is of the form given by
\[
u(x)
= \Bigg(\prod_{i=n-k+1}^{n} x_i\Bigg)
\, e^{-B{|x|^2}}
\Bigg(\sum_{i=1}^{n-k} b_i x_i + b
\Bigg),
\qquad b_1,\dots,b_{n-k},b\in\mathbb{R}, \text{ and } B>0.
\] 
\end{theorem}

Our paper is organized as follows: In Section 2, we will provide some important calculus on orthants and prove some useful lemmas. In Section 3, we will prove Theorem \ref{T:HeisOrth} and determine the optimal constant as well as all possible optimizers for the HUP on orthants. Some stability estimates for the HUP on orthants will be established in Section 4.
  
\section{Integration formulas on Orthants}

Recall that by an {\em orthant} we mean subset of $\R^n$ of the form
$$
    \R_{k,+}^n:=\R^{n-k}\times \R_{>0}^k,
$$
where $k\in\{1,\ldots, n\}$ and $\R_{>0}=(0,\infty)$.  By a ``wall'' of this orthant we mean a set of the form
$$
    \{x\in \overline{\R_{k,+}^n}\,:\, x_i=0\}
$$
for some $i\in \{n-k+1,\ldots, n\}$. In more general notation,
for $A\subset\{1,\ldots, n\}$, we have an orthant
\begin{equation}\label{E:RnAp}
    \RnAp:=\{x\in\R^n\,:\,\hbox{$ x_i>0$ for $i\in A$}\}.
\end{equation}

In this section, we  will establish some integration formulas for functions on $\Rnkp$ that will be used in converting inequalities and identities on $\Rnkp$ to corresponding relations on a full Euclidean space, and vice versa.


We display a point $x\in\mathbb{R}^n$ as $x=(x',x'')$ with $x'\in\mathbb{R}^{\,n-k}$ and $x''\in\mathbb{R}^{\,k}$; on the orthant we thus have $x''>0$ componentwise; in this notation,
$$x''_i=x_{n-k+i}.$$
The key strategy (from  \cite[Proof of Lemma 2.1]{SuYang2012})  is to obtain $\Rnkp$ as an image of the full Euclidean space  in the following way:
$$
    P:  \R^{n-k}\times\R^{2l_1+1}\times\ldots\times \R^{2l_k+1}   \to {\Rnkp}: (x', y_1,\ldots, y_k)\mapsto (x', |y_1|,\ldots,|y_k|),
$$
where  $y_j=(y_{j,1},\ldots, y_{j, 2l_j+1})$ has norm:
$$|y_j|:= \sqrt{y_{j,1}^2+\ldots +y_{j, 2l_j+1}^2}.$$
We need each $2l_j$ to be a positive integer, and we write
$$l=(l_1,\ldots, l_k)\quad\hbox{and}\quad |l|=l_1+\ldots +l_k.$$
We  lift a given function $v$ on $\Rnkp$ to the function $\tlv$ on $\R^{n+2|l|}$ given by 
\begin{equation}\label{E:deftlv}
 \tlv\bigl(x',y\bigr):= (v\circ P)(x',y)=v(x',|y_1|,\dots,|y_k|).
\end{equation}
 In addition, we will frequently use the function
\begin{equation}\label{E:defuxvx}
    u(x)=\left(\prod_{i=n-k+1}^{n} x_i^{l_i}\right)\, v(x)\qquad\text{for\, } x\in\Rnkp.
\end{equation}
The notation may be generalized as follows. Let $A\subset\{1,\ldots, n\}$, and choose half-integers $l_i\geq 0 $, with $l_i=0$ for $i\notin A$. Then  we have a covering map
\begin{equation}\label{E:defPA}
    P_A:\R^{n-|A|}\times\prod_{i\in A}\R^{2l_i+1}\to\RnAp: x\mapsto P_Ax,
\end{equation}
where $(P_Ax)_i=x_i$ if $i\notin A$ and $(P_Ax)_i=|x_i|$ if $i\in A$.

We can state the main integration formulas now.

\begin{theorem}\label{T:integration}
Let $v$ be a measurable function on $\Rnkp$, and let $\tlv$ and $u$ be defined as above. Then
\begin{equation}\label{E:inttlvx'y}
\begin{split}
\int_{\R^{n+2|l|}} |\tlv(x', y)|^2\,dx' dy 
          &= \left(\prod_{i=n-k+1}^n|\mathbb{S}^{2l_i}|\right)\int_{\Rnkp}|v(x)|^2\left(\prod_{i\in\{n-k+1,\ldots, n\}}x_i^{2l_i}\right)\,dx\\
          &=\left(\prod_{i=n-k+1}^n|\mathbb{S}^{2l_i}|\right)\int_{\Rnkp}|u(x)|^2 \,dx.
          \end{split}
          \end{equation}
    More generally,
    \begin{equation}\label{E:inttlvx'yT}
 \int_{\R^{n+2|l|}} \left(|x'|^2+\sum_{i=1}^k|y_i|^2\right)^{a}|\tlv(x',y)|^2\,dx' dy  =\left(\prod_{i=n-k+1}^n|\mathbb{S}^{2l_i}|\right)\int_{\Rnkp} |x|^{2a}|u(x)|^2 \,dx,
          \end{equation}
          for any $a\in\R$.
          Now suppose $v\in C^\infty_c(\Rnkp)$, vanishing in a neighborhood of the walls of $\Rnkp$. Then
         \begin{equation}\label{E:intnabtlv2nabluT}
      \begin{split}
          &\left(\prod_{i=n-k+1}^n|\mathbb{S}^{2l_i}|\right)^{-1}\int_{\mathbb{R}^{n+2|l|}} |x|^{2b} |\nabla \tlv(x',y)|^2  \, dx' \, dy \\
     & =\int_{\Rnkp}\left[|x|^{2b}|\nabla u|^2+\sum_{i=n-k+1}^n\frac{l_i(l_i-1)}{x_i^2} |u(x)|^2+2b|l||x|^{2b-2}|u(x)|^2\right]\,dx,
      \end{split}
  \end{equation}
  for all $b\in\R$.
\end{theorem}
We note, for later use, that the set $\{n-k+1,\ldots, n\}$ can be replaced by any subset $A\subset\{1,\ldots, n\}$, with natural adjustment of the formulas above. We will use this more general notation in sections  \ref{s:GPImon} and \ref{s:StabHUPOrth}. In particular, replacing $v(x)$ by the function $v(x)$ in (\ref{E:inttlvx'y}), we have
\begin{equation}\label{E:inttlvx'yA}
    \int_{\R^{n+2|l|}}|\tlv(y)|^2\,dy =\prod_{i\in A}|\mathbb{S}^{2l_i}|\int_{\RnAp}|v(x)|^2x^{2l}\,dx,
\end{equation}
where $l=(l_1,\ldots, l_n)$, with each $l_i$ being a non-negative half-integer and $l_j$ being $0$ for $j\notin A$, and $|l|=\sum_{i=1}^nl_i$

\begin{proof}
From the definition of $\tlv$ and using polar coordinates for integration in each variable $y_i\in\R^{2l_i+1}$, we have:
\begin{align}
        \int_{\R^{n+2|l|}} |\tlv(x', y)|^2\,dx' dy &=  \int_{\Rnkp}|v(x', r_1,\ldots, r_k)|^2\, \prod_{i=n-k+1}^n\left(|\mathbb{S}^{2l_i}| r_i^{2l_i}\,dr_i\right)\, dx' \nonumber\\
        &=  \left(\prod_{i=n-k+1}^n|\mathbb{S}^{2l_i}|\right)\int_{\Rnkp}|v(x)|^2\prod_{i=n-k+1}^nx_i^{2l_i}\,dx \nonumber\\
          &=  \left(\prod_{i=n-k+1}^n|\mathbb{S}^{2l_i}|\right)\int_{\Rnkp}|u(x)|^2 \,dx \label{E:inttlvx'y0}
\end{align}
This proves (\ref{E:inttlvx'y}).

Given a function on $v$ on $\Rnkp$.  We define the function $w$ on $\Rnkp$ given by
$$w(x)=|x|^{a}v(x).$$
Then we have, in place of $\tlv$,  the function ${\tilde w}$ on $\mathbb{R}^{n+2|l|}$  given by:
$${\tilde w}(x',y)= w(x', |y_1|,\ldots, |y_k|)=\left(|x'|^2+\sum_{i=1}^k|y_i|^2\right)^{a/2}\tlv(x',y).$$
Then, applying (\ref{E:inttlvx'y0}), we obtain:
\begin{equation}\label{E:inttlvx'ya0}
    \begin{split}
        &\int_{\R^{n+2|l|}} \left(|x'|^2+\sum_{i=1}^k|y_i|^2\right)^{a}|\tlv(x',y)|^2\,dx' dy \\
        & =  \left(\prod_{i=n-k+1}^n|\mathbb{S}^{2l_i}|\right)\int_{\Rnkp}\Bigl(\prod_{i=n-k+1}^{n} x_i^{l_i}\Bigr)^2|x|^{2a}|v(x)|^2 \,dx\\
        & =  \left(\prod_{i=n-k+1}^n|\mathbb{S}^{2l_i}|\right)\int_{\Rnkp} |x|^{2a}|u(x)|^2 \,dx,
    \end{split}
\end{equation}
with $u(x)$ as in (\ref{E:defuxvx}). This proves (\ref{E:inttlvx'yT}).

Now we work with $v\in C^\infty_c(\Rnkp)$, the function $\tlv$ on $\R^{n+2|l|}$ defined by
       \begin{equation}\label{E:tlvy}
       \begin{split}
            \tlv(x',y) &=v(x', |y_1|,\ldots, |y_k|) \\
            & \qquad\hbox{for all $(x',y)=(x',y_1,\ldots, y_k)\in  \R^{n+2|l|}=\R^{n-k}\times \prod_{i=1}^k\R^{2l_i+1}$,} 
            \end{split}
       \end{equation}
       and the function $u$ on $\Rnkp$ defined by
       \begin{equation}\label{E:defutlv}
       u(x)=\left(\prod_{i=1}^kx_{n-k+i}^{l_i}\right)v(x)\qquad\hbox{for all $x\in\Rnkp$.}
       \end{equation}
Since $v\in C^\infty_c(\Rnkp)$, it has compact support in the orthant $\Rnkp$, and so the value $v(x', |y_1|,\ldots, |y_k|)$ is $0$ whenever $\max_{i}|y_i|$ is small; hence, $\tlv(x',y)$ is  $0$ when $y$ is close to $0$ in $\R^{2|l|}$. Thus, the function $\tlv$ is in $C^\infty_c(\R^{n+2l})$ and is $0$ near $\R^{n-k}\times \{0\}\subset \R^{n}\times \R^{2|l|}$.


Next we compute $\Delta u(x)$ in relation to $\Delta v(x)$:
\begin{align}
    \Delta u(x)  & =\Delta\left(v(x)\prod_{i=n-k+1}^nx_i^{l_i}\right) \nonumber\\
       &=\left(\Delta v(x)\right)\prod_{i=n-k+1}^nx_i^{l_i}+2\sum_{i=n-k+1}^n\frac{\partial v(x)}{\partial x_i}\frac{l_i}{x_i}\prod_{j=n-k+1}^nx_j^{l_j}\nonumber\\
        &\qquad\qquad + v(x) \sum_{i=n-k+1}^n\frac{l_i(l_i-1)}{x_i^2}\prod_{j=n-k+1}^nx_j^{l_j}\nonumber\\
        &=\left[\Delta v(x)+ \sum_{i=n-k+1}^n\left(\frac{2l_i}{x_i}\frac{\partial v(x)}{\partial x_i}+ \frac{l_i(l_i-1)}{x_i^2}v(x)\right)\right]\prod_{j=n-k+1}^nx_j^{l_j}. \label{E:Deltaux}
\end{align}

We compute $\Delta \tlv$. Recall that $\tlv(x',y)$ is a radial function with respect to each variable $y_i$, namely depending on each $y_i$ through $|y_i|$. So we need to use the Laplacian for a radial function:
$$
    \Delta_y f(|y|)=f''(|y|)+\frac{d-1}{|y|}f'(|y|),
$$
where $d$ is the dimension of the space over which $y$ runs. Thus, with $x=(x',x_{n-k+1},\ldots, x_n)=P(x',y)$,
$$
 \Delta {\tlv} (x',y) = \Delta v(x) + \sum_{i=n-k+1}^n\frac{2l_i}{x_{ i}}\frac{\partial v(x)}{\partial  x_{ i}}.
$$
Using this in (\ref{E:Deltaux}), we obtain:
$$
        \Delta u(x) =\left[\Delta \tlv(x',y)+\sum_{i=n-k+1}^n\frac{l_i(l_i-1)}{x_i^2}v(x)\right]\left(\prod_{j=n-k+1}^nx_j^{l_j}\right),
    $$
where $x_{n-k+i}=|y_i|$, 
and this yields:
\begin{equation}\label{E:Delttlvxpyu}
    \begin{split}
        \Delta {\tlv}(x',y) &=\left(\prod_{j=n-k+1}^nx_j^{l_j}\right)^{-1} \Delta u(x) - \sum_{i=n-k+1}^n\frac{l_i(l_i-1)}{x_i^2}v(x)\\
        &=\left(\prod_{j=n-k+1}^nx_j^{l_j}\right)^{-1}\left[\Delta u(x) - \sum_{i=n-k+1}^n\frac{l_i(l_i-1)}{x_i^2}u(x)\right].
    \end{split}
\end{equation}

Next, again with the variable $x_{n-k+i}$ corresponding to $|y_i|$ when switching to polar coordinate integration, and $f=f(x', |y_1|,\ldots, |y_k|)$ any measurable function on $\Rnkp$ for which the integrals below are defined, we have, with $x=(x',|y_1|,\ldots, |y_k|)$,
\begin{align*}
     &  \int_{\R^{n+2|l|}} f(x){\tlv}(x',y)\Delta {\tlv}(x',y)\,dx'\,dy  \\
       & = \int_{\Rnkp} f(x)v(x) \left(\prod_{j=n-k+1}^nx_j^{l_j}\right)^{-1}\left[\Delta u(x) - \sum_{i=n-k+1}^n\frac{l_i(l_i-1)}{x_i^2}u(x)\right] \cdot \\
       &\hskip 2in \cdot \left(\prod_{i=n-k+1}^n|\mathbb{S}^{2l_i}|x_i^{2l_i}dx_i\right)    \,dx \\
        & = \prod_{i=n-k+1}^n|\mathbb{S}^{2l_i}|\int_{\Rnkp} f(x)u(x) \left[\Delta u(x) -\sum_{i=n-k+1}^n\frac{l_i(l_i-1)}{x_i^2}u(x)\right] \cdot \\
       &\hskip 2in \cdot \left(\prod_{i=n-k+1}^n dx_i\right)    \,dx'\\
       &=  \left(\prod_{i=n-k+1}^n|\mathbb{S}^{2l_i}|\right)\int_{\Rnkp} f(x)u(x) \left[\Delta u(x) -\sum_{i=n-k+1}^n\frac{l_i(l_i-1)}{x_i^2}u(x)\right] \,dx.
\end{align*}
Thus
\begin{equation}\label{E:intftlvDtlv}
\begin{split}
\int_{\R^{n+2|l|}} f(x){\tlv}(x',y)\Delta {\tlv}(x',y)\,dx'\,dy  &\\
        &\hskip -2.25in =  \left(\prod_{i=n-k+1}^n|\mathbb{S}^{2l_i}|\right)\int_{\Rnkp} f(x)u(x) \left[\Delta u(x) -\sum_{i=n-k+1}^n\frac{l_i(l_i-1)}{x_i^2}u(x)\right] \,dx.
     \end{split}
  \end{equation}
Note that we only used the transition to polar coordinates, and no integration by parts was used in this calculation. 

We set $f=1$.  Since $v$ vanishes near the walls of $\Rnkp$, we integrate by parts, to obtain, after rearranging terms,
\begin{equation}\label{E:intnablau2}
\begin{split}
& \int_{\R^{n+2|l|}} |\nabla {\tilde v}(x)|^2\,dx  \\
       &\hskip .5in  =\left(\prod_{i=n-k+1}^n|\mathbb{S}^{2l_i}|\right)\int_{\Rnkp}\left[|\nabla u|^2+\sum_{i=n-k+1}^n\frac{l_i(l_i-1)}{x_i^2}|u(x)|^2\right] \,dx.
        \end{split}
       \end{equation}
Next we have, as always with $x_{n-k+i}$ replacing $|y_i|$,
\begin{align*}
        &\int_{\mathbb{R}^{n-k} \times \prod_{i=1}^k\mathbb{R}^{2l_i+1}} |x|^{2b} |\nabla \tlv(x',y)|^2  \, dx' \, dy\\
    &\hskip 1in \qquad\hbox{(where $x=(x', |y_1|,\ldots, |y_k|)\in\Rnkp$)}\\
     &= -\int_{ \mathbb{R}^{n+2|l|}} 
       {\tlv(x',y)}\nabla_{(x',y)}\cdot\left(|x|^{2b} \nabla {\tlv}_{(x',y)}(x',y)\right)\,dx' dy   \\
    &= -\int_{ \mathbb{R}^{n+2|l|}} 
       {\tlv(x',y)} \left[{|x'|^{2b}}\Delta \tlv(x',y)+ 2b|x|^{2b-1}\frac{(x',y)}{|x|}\cdot\nabla \tlv(x',y)\right]\, dx' \, dy  \\
       &\hskip 2in \hbox{(note that $|x|=|(x',y)|$)} \\
    &= -\int_{\mathbb{R}^{n+2|l|}} 
     {\tlv(x',y)}{|x'|^{2b}}  \Delta \tlv(x',y)\,dx'dy \\
     &\hskip 2in -b\int_{\R^{n+2|l|}}|x|^{2b-2}(x',y)\cdot 2{\tlv}(x',y)\nabla \tlv(x',y)\,dx'dy\\
     &= -\int_{\mathbb{R}^{n+2|l|}} 
     {\tlv(x',y)}{|x'|^{2b}}  \Delta \tlv(x',y)\,dx'dy \\
     &\hskip 2in -b\int_{\R^{n+2|l|}}|x|^{2b-2}(x',y)\cdot \nabla \left\vert\tlv(x',y)\right\vert^2\,dx'dy\\
     &= -\int_{\mathbb{R}^{n+2|l|}} 
     {\tlv(x',y)}{|x'|^{2b}}  \Delta \tlv(x',y)\,dx'dy \\
     &\hskip 2in +b\int_{\R^{n+2|l|}}\left[\nabla\cdot\left(|x|^{2b-2}(x',y)\right)\right]\left\vert\tlv(x',y)^2\right\vert\,dx'dy.\\
   \end{align*}
     Continuing further, we have
     \begin{align*}
          &\int_{\mathbb{R}^{n-k} \times \prod_{i=1}^k\mathbb{R}^{2l_i+1}} |x|^{2b} |\nabla \tlv(x',y)|^2  \, dx' \, dy \\
          &= -\int_{\mathbb{R}^{n+2|l|}} 
     {\tlv(x',y)}{|x'|^{2b}}  \Delta \tlv(x',y)\,dx'dy \\
           &\hskip .25in +b\int_{\R^{n+2|l|}}\left[(2b-2)|x|^{2b-3}\frac{(x',y)}{|x|}\cdot (x',y) +|x|^{2b-2}\nabla_{(x',y)}\cdot(x',y)\right]\cdot\\
     &\hskip 4in \cdot\left\vert\tlv(x',y)^2\right\vert\,dx'dy\\
      &\hskip .25in= -\int_{\mathbb{R}^{n+2|l|}} 
     {\tlv(x',y)}{|x'|^{2b}}  \Delta \tlv(x',y)\,dx'dy \\
     &\hskip 1in +b\int_{\R^{n+2|l|}}|x|^{2b-2}\left[(2b-2) +n+2|l|\right]\left\vert\tlv(x',y)\right\vert^2\,dx'dy.
   \end{align*}
We can now switch to polar coordinates in the variables $y_i$, since the integrand at the end depends on $y$ through $(|y_1|,\ldots, |y_k|)$. We have then
\begin{align*}
         &\int_{\mathbb{R}^{n-k} \times \prod_{i=1}^k\mathbb{R}^{2l_i+1}} |x|^{2b} |\nabla \tlv(x',y)|^2  \, dx' \, dy \\
     &  = - 
     \left(\prod_{i=n-k+1}^n|\mathbb{S}^{2l_i}|\right)\int_{\Rnkp} |x|^{2b}u(x) \left[\Delta u(x) -\sum_{i=n-k+1}^n\frac{l_i(l_i-1)}{x_i^2}u(x)\right] \,dx \\
     &\hskip 3in \quad\hbox{(by (\ref{E:intftlvDtlv})}) \\
     &\hskip 1in +\left(\prod_{i=n-k+1}^n|\mathbb{S}^{2l_i}|\right) b\left[(2b-2) +n+2|l|\right]\int_{\Rnkp}|x|^{2b-2} |u(x)|^2 \,dx \\
     &\hskip 3.5in\qquad\hbox{(using (\ref{E:inttlvx'yT})).}
\end{align*}
Thus,
\begin{equation}\label{E:x2bnabtlv}
    \begin{split}
         & \left(\prod_{i=n-k+1}^n|\mathbb{S}^{2l_i}|\right)^{-1}\int_{\mathbb{R}^{n+2|l|}} |x|^{2b} |\nabla \tlv(x',y)|^2  \, dx' \, dy  \\
     & \hskip 1in = -\int_{\Rnkp} |x|^{2b}u(x) \left[\Delta u(x) -\sum_{i=n-k+1}^n\frac{l_i(l_i-1)}{x_i^2}u(x)\right] \,dx   \\
     &\hskip 2in +b(n+2b-2+2|l|)\int_{\Rnkp}|x|^{2b-2}  |u(x)|^2 \,dx.
    \end{split}
\end{equation}
Now we will run the same procedure to convert $u\Delta u$ to $|\nabla u|^2$:
\begin{align*}
        &\int_{\Rnkp}|x|^{2b}u(x)\Delta u(x)\, dx \\
        &= -\int_{\Rnkp}\nabla\left(|x|^{2b}u(x)\right)\cdot\nabla u\,dx\\
        &=-\int_{\Rnkp}\left[\left\{2b|x|^{2b-1}\frac{x}{|x|}u(x)\right\}\cdot\nabla u +|x|^{2b}|\nabla u|^2\right]\,dx\\
        &= -b\int_{\Rnkp}|x|^{2b-2}x\cdot\nabla\left\vert u(x)\right\vert^2\,dx -\int_{\Rnkp}|x|^{2b}|\nabla u|^2\,dx.
         \end{align*}
Integrating by parts, we obtain:
\begin{equation}
    \begin{split}
        &=b\int_{\Rnkp}\nabla\cdot\left(|x|^{2b-2}x\right) \,|u(x)|^2\,dx -\int_{\Rnkp} |x|^{2b}|\nabla u|^2\,dx\\
        &=b\int_{\Rnkp}\left[(2b-2)|x|^{2b-3}\frac{x}{|x|}\cdot x +|x|^{2b-2}\nabla \cdot x\right] |u(x)|^2\,dx \\
        &\hskip 2in -\int_{\Rnkp}|x|^{2b}|\nabla u|^2\,dx\\
        &=b\int_{\Rnkp}\left[(2b-2)|x|^{2b-2}+n|x|^{2b-2}\right]|u(x)|^2\,dx -\int_{\Rnkp}|x|^{2b}|\nabla u|^2\,dx.
    \end{split}
\end{equation}

To summarize,

\begin{equation}\label{E:intuDudelu}
              \int_{\Rnkp}|x|^{2b}u(x)\Delta u(x)\, dx  =b(n+2b-2)\int_{\Rnkp}|x|^{2b-2}|u(x)|^2\,dx -\int_{\Rnkp}|x|^{2b}|\nabla u|^2\,dx.
       \end{equation}
  Using this all the way back in (\ref{E:x2bnabtlv}), we obtain:

\begin{equation}\label{E:intnabtlv2nablu}
 \begin{split}
          &\left(\prod_{i=n-k+1}^n|\mathbb{S}^{2l_i}|\right)^{-1}\int_{\mathbb{R}^{n+2|l|}} |x|^{2b} |\nabla \tlv(x',y)|^2  \, dx' \, dy \\
     & =\int_{\Rnkp}\left[|x|^{2b}|\nabla u|^2+\sum_{i=n-k+1}^n\frac{l_i(l_i-1)}{x_i^2} |u(x)|^2+2b|l||x|^{2b-2}|u(x)|^2\right]\,dx.
      \end{split}
  \end{equation}
  This proves (\ref{E:intnabtlv2nabluT}) and completes the proof of all the identities. \end{proof}
  
  The significance of the formula (\ref{E:intnabtlv2nablu})  is that it converts the integral of the quadratic gradient integral of $u$ on the orthant $\Rnkp$ to a corresponding integral on the full Euclidean space $\R^{n+2|l|}$. This will allow us to translate integral identities and inequalities on the Euclidean space into inequalities for the orthant.

\begin{lemma}\label{II}
    Let $v\in C^\infty_c(\Rnkp)$, and let $\tlv$ be the function on $\R^{n-k}\times(\R^3)^k$ given by
    \begin{equation*}
    \begin{split}
        {\tlv}(x',y_1,\ldots, y_k)&:= v(x', |y_1|,\ldots, |y_k|)\\
        &\hbox{for $x'\in\R^n$ and $y_1,\ldots, y_k\in\R^3$,}
        \end{split}
    \end{equation*}
    and let
    \begin{equation*}
        u(x):=\left(\prod_{i=n-k+1}^{n} x_i\right)\, v(x)\qquad\hbox{for $x\in\Rnkp$.}
    \end{equation*}  
    Then: 
\begin{equation}\label{E:gradtlvgradu1}
\begin{split}
     \int_{\mathbb{R}_{x'}^{n-k} \times \mathbb{R}_y^{3k}} |\nabla \tlv(x',y)|^2 \; dx' dy &= |\mathbb{S}^2|^k \int_{\Rnkp} 
    \left|\nabla\!\left(\frac{u(x)}{\prod_{i=n-k+1}^n x_i}\right)\right|^2
\prod_{i=n-k+1}^n x_i^2 \, dx  \\
&=|\mathbb{S}^2|^k \int_{\Rnkp} 
    \left|\nabla\!u(x)\right|^2 \, dx,
\end{split}
\end{equation}
and,
\begin{equation}\label{E:gradtlvgradu2}
\begin{split}
&\int_{\mathbb{R}_{x'}^{n-k} \times \mathbb{R}_y^{3k}} \tlv(x',y)\left((x',y)\cdot \nabla \tlv(x',y)\right) \; dx' dy\\
&=|\mathbb{S}^2|^k \int_{\Rnkp}\left(\frac{u(x)}{\prod_{i=1}^k x_i}\right) \left[x\cdot \nabla\!\left(\frac{u(x)}{\prod_{i=n-k+1}^n x_{i}}\right)\right] \prod_{i=n-k+1}^nx_i^2 \; dx.
\end{split}
\end{equation}

\end{lemma}
\begin{proof} We begin with the right side of (\ref{E:gradtlvgradu1}):

\begin{align*}
      &\int_{\Rnkp} \left|\nabla\left(\frac{u(x)}{\prod_{i=n-k+1}^n x_i}\right)\right|^2\prod_{i=n-k+1}^n x_i^2 \, dx\\
&= \int_{\Rnkp} \left|\frac{\nabla u}{\prod_{i=n-k+1}^n x_i} 
  + \frac{u(x)}{\prod_{i=n-k+1}^n x_i}  \left(0,\ldots,0,-\tfrac{1}{x_{n-k+1}},\ldots,-\tfrac{1}{x_n}\right)\right|^2\prod_{i=n-k+1}^n x_i^2 \, dx\\ 
&= \int_{\Rnkp} |\nabla u|^2 \, dx
   + \int_{\mathbb{R}^k_+} 
     u^2 \left(\sum_{i=n-k+1}^n \frac{1}{x_i^2}\right) dx
   + 2\int_{\mathbb{R}^k_+} 
     u \left(\sum_{i=n-k+1}^n \frac{\partial u}{\partial x_i}\cdot\frac{-1}{x_i}\right) dx\\
&= \int_{\mathbb{R}^k_+} |\nabla u|^2 \, dx
   + \sum_{i=n-k+1}^n \int_{\mathbb{R}^k_+} 
     \left(\frac{u^2}{x_i^2} - \frac{2u}{x_i}\frac{\partial u}{\partial x_i}\right) dx.
    \end{align*}
Continuing further,
\begin{align*}
     & \int_{\Rnkp} \left|\nabla\!\left(\frac{u(x)}{\prod_{i=1}^k x_i}\right)\right|^2\prod_{i=1}^k x_i^2 \, dx\\
&= \int_{\Rnkp} |\nabla u|^2 \, dx
   - \sum_{i=n-k+1}^n \int_{\Rnkp} 
     \frac{\partial}{\partial x_i}\!\left(\frac{u^2}{x_i}\right) dx\\
&= \int_{\Rnkp} |\nabla u|^2 \, dx \\
&=|\mathbb{S}^2|^{-k} \int_{\mathbb{R}_{x'}^{n-k} \times \mathbb{R}_y^{3k}} |\nabla \tlv(x',y)|^2 \; dx' dy
\qquad \text{(using (\ref{E:intnabtlv2nabluT}) for all } l_i=1).
\end{align*}
This establishes the identity (\ref{E:gradtlvgradu1}).

For the second identity, we have
\begin{align*}
       & \int_{\mathbb{R}_{x'}^{n-k} \times \mathbb{R}_y^{3k}}\tlv(x',y)((x',y)\cdot \nabla \tlv(x',y)) \; dx dy\\
       &=  \int_{\mathbb{R}_{x'}^{n-k} \times \mathbb{R}_y^{3k}}\tlv(x',y)\left(\sum_{i=1}^{n-k} x_i\frac{\partial \tlv}{\partial x_i}+\sum_{i=1}^{3k} y_i\frac{\partial \tlv}{\partial y_i}\right) \, dx' dy\\
    &= |\mathbb{S}^2|^k \int_{\Rnkp} v(x)\left(\sum_{i=1}^nx_i\frac{\partial v}{\partial x_i}\right)\prod_{i=n-k+1}^nx_i^2 dx\\
    &\hskip 1in \hbox{(passing to ``polar'' coordinates)}\\
    &=|\mathbb{S}^2|^k \int_{\Rnkp}\left(\frac{u(x)}{\prod_{i=1}^k x_i}\right) x\cdot \nabla\!\left(\frac{u(x)}{\prod_{i=1}^k x_i}\right) \prod_{i=n-k+1}^nx_i^2 dx.
    \end{align*}
This completes the proof of (\ref{E:gradtlvgradu2}).
\end{proof}

\subsection{Functions that vanish on the orthant walls}

We begin with an observation about smooth functions on the closure of an orthant that vanish on the walls of the orthant. Recall that, by the Whitney extension theorem, a $C^\infty$ function on a closed subset of $\R^n$ has a $C^\infty$ extension to a neighborhood of the set.

\begin{lemma}
Let $U$ be the  orthant $\R^{n-k}\times\R^k_+$, and let
  $u\in C^\infty(\bar{U})$ be such that
\[
u(x)=0 \quad  \text{ if $x_j=0$ for some } j\in J:=\{n-k+1,\ldots, n\}.
\tag{$\star$}
\]
Then there exists $v\in C^\infty(\bar{U})$ such that
\begin{equation}\label{E:uprodbaru}
    u(x)=\Big(\prod_{j\in J} x_j\Big)\,v(x)\qquad\text{for all }x\in \bar{U}.
\end{equation}
If  the support of ${u}$ is contained in $U$, then it is equal to the support of $v$. 
\end{lemma}

\begin{proof} Let us first establish the statement about supports. A point is not in the support of a function if and only if it has a neighborhood on which the function is $0$. If $p$ is not in $\supp(u)$ then $p$ is  a point on a wall of $U$ or is not in $\supp(u)$. If $\supp(u)\subset U$, then it does not contain any point of any wall, and so a point is outside $\supp(u)$ if and only if it is outside $\supp(\bar{u})$.

We turn now to proving (\ref{E:uprodbaru}). To begin, we consider the case $J=\{n\}$. We will show that $x\mapsto u(x)/x_n$ is $C^\infty(\bar{U})$.
Fix $(x',x_n)\in U$. Then:
\[
u(x',x_n)-u(x',0)=\int_0^1 \frac{d}{dt}\,u(x',t x_n)\,dt
=\int_0^1 \partial_{x_n}u(x',t x_n)\,x_n\,dt.
\]
Since $u(x',0)=0$ by hypothesis, we obtain
\[
u(x',x_n)=x_n \int_0^1 \partial_{n}u(x',t x_n)\,dt.
\]
Define
\[
v(x',x_n):=\int_0^1 \partial_{n}u(x',t x_n)\,dt.
\]
Then $u=x_n v$ pointwise on $U$. The expression for $v$ shows that it is $C^\infty$. 
\end{proof}

    \begin{lemma}
If  $\bar{u} \in C_c^\infty(\mathbb{R}^{n,k}_+)$, then the function $v$ on $\R^{n-k}\times \R^{3k}$
given by
\[v(x',y) = \bar{u}\!\left(x', r_1, \dots, r_k\right), \]
$
\quad r_j := \|(y_{3j-2}, y_{3j-1}, y_{3j})\|=\sqrt{y_{3j-2}^2 + y_{3j-1}^2 + y_{3j}^2},\quad j\in\{1,\ldots, l\}
$
is in $C_c^\infty(\mathbb{R}^{n-k} \times \mathbb{R}^{3k})$.
\end{lemma}

\begin{proof} \small{
Consider the continuous mapping
\[
P : \mathbb{R}^{n-k} \times \mathbb{R}^{3k} \to \mathbb{R}^n,
\qquad (x',y) \mapsto (x', r_1,\dots,r_k).
\]
The preimage under $P$ of a closed set is closed, and the preimage 
of a bounded set is also bounded. Therefore, the preimage of a 
compact set is compact. Since $\operatorname{supp}(\tilde{u})$ 
is compact, it follows that $\operatorname{supp}(v)$ is compact.

Moreover, because $\tilde{u}(x)=0$ whenever some $x_i < \delta$ 
for a sufficiently small $\delta>0$, we distinguish two cases. If $r_j < \delta$ for some $j$, then $v=0$ in a neighborhood, hence $v$ is smooth there.  
And If $r_j \geq \delta$ for all $j$, then the map $y^{(j)} \mapsto r_j$ is $C^\infty$, and composing this smooth 
map with $\tilde{u}$ preserves smoothness.  
Thus $v \in C_c^\infty(\mathbb{R}^{n-k} \times \mathbb{R}^{3l})$.}
\end{proof}
   
\section{The Heisenberg Uncertainty Principle on an orthant: Proof of Theorem \ref{T:HeisOrth}}

In this section, we will establish a Heisenberg Uncertainty Principle for functions on $\R_{k,+}^n$ that vanish on the  walls  of the orthant. We will show, in Theorem \ref{T:HeisOrth},  that
\begin{equation}\label{E:HeisOrth0}
\int_{{\Rnkp}} |\nabla u(x)|^2 \, dx  \; \int_{{\Rnkp}} |x|^2 |u(x)|^2 dx
\geq \frac{(n + 2k)^2}{4} \left( \int_{{\Rnkp}} |u(x)|^2 \, dx\right)^2  
\end{equation}
holds for all $u\in S(\R_{k,+}^n)$. 

Unlike the  Heisenberg Uncertainty Principle  on $\R^n$, inequality (\ref{E:HeisOrth0}) is not simply a direct consequence of corresponding inequalities for the (essentially one-dimensional) partial derivatives; this is seen in the fact that the constant in (\ref{E:HeisOrth0}) is higher than the one that applies to all $u\in C^\infty_c(\mbr^n)$. 

\begin{proof}[Proof of Theorem \ref{T:HeisOrth}.]
    By a density argument, we can assume that $u \in C_c^\infty(\mathbb{R}^n_{k, +})$.  Then there exists ${v} \in  C_c^\infty({\Rnkp})$ such that \begin{equation*}
        u(x)=\left(\prod_{i=n-k+1}^{n} x_i\right)    v(x)\qquad\hbox{  for all $x\in\Rnkp$.}
    \end{equation*}
    We lift $v$ to a function on $\R^{n-k}\times (\R^3)^k$  given by
    \begin{equation*}
        \tlv(x',y)=v(x', |y_1|,\ldots, |y_k|)
    \end{equation*}
    for all $x'\in\R^{n-k}$ and with each $y_i$ running over $\R^3$. Note that since $v$ is smooth, of compact support, and vanishes near the boundary of $\Rnkp$, we have $\tlv\in C^\infty_c\bigl(\R^{n-k}\times (\R^3)^k\bigr)$.

    Applying the Heisenberg Uncertainty Principle \eqref{HUP} for $\R^{n-k}\times (\R^3)^k$ to $\tlv$  we have:
    \begin{equation}\label{E:HUPtlv}
    \begin{split}
        &\int_{\mathbb{R}^{n-k} \times\mbr^{3k}} |\nabla \tlv(x',y)|^2 \, dx'dy  \; \int_{\mathbb{R}^{n-k} \times\mbr^{3k}} |(x',y)|^2 |\tlv(x',y)|^2 dx'dy\\
    &\hskip 1in\geq \frac{(n + 2k)^2}{4} \left( \int_{\mathbb{R}^{n-k} \times\mbr^{3k}} |\tlv(x',y)|^2 \, dx' dy\right)^2,
    \end{split}
    \end{equation}
    where we   use the notation
    $$(x',y)=(x',y_1, \ldots, y_k).$$
Recall from (\ref{E:inttlvx'yT}) and  (\ref{E:intnabtlv2nablu}), with each $l_i$ equal to $1$ and $b=0$, that
\begin{equation}\label{E:inttlvx'y2}
 \int_{\R^{n+2k}} \left(|x'|^2+\sum_{i=1}^k|y_i|^2\right)^{a}|\tlv(x',y)|^2\,dx' dy  =\left(\prod_{i=n-k+1}^n|\mathbb{S}^{2}|\right)\int_{\Rnkp} |x|^{2a}|u(x)|^2 \,dx.
          \end{equation}
          and
\begin{equation}\label{E:intnabtlv2nablu2}
 \begin{split}
          \int_{\mathbb{R}^{n+2k}}  |\nabla \tlv(x',y)|^2  \, dx' \, dy &=\left(\prod_{i=n-k+1}^n|\mathbb{S}^{2}|\right)\int_{\Rnkp}\left\vert \nabla u  \right\vert^2\,dx.
      \end{split}
  \end{equation}
  Using these formulas in (\ref{E:HUPtlv}),  with each $l_i$ being $1$, $b=0$ and $a=1$ (as well as the case $a=0$), we obtain, after cancelling the terms involving $|\mathbb{S}^2|$,
  \begin{equation}\label{E:HUPorth2}
      \int_{\Rnkp}|\nabla u|^2\,dx \,\int_{\Rnkp}|x|^2|u(x)|^2\,dx \geq \left(\frac{n+2k}{2}\right)^2\left(\int_{\Rnkp}|u(x)|^2\,dx\right)^2.
  \end{equation}
  Indeed, equality holds in (\ref{E:HUPorth2}) if and only if it holds in the original Heisenberg Uncertainty Principle  (\ref{E:HUPtlv}) on $\R^{n+2k}$. The latter holds if and only if  
  \begin{equation*}
       \tlv(x)=\alpha e^{-\frac{\beta}{2}|x|^2},
  \end{equation*}
  for some constants $\alpha$ and $\beta>0$, and so
  \begin{equation*}
      u(x)= \alpha\Big(\prod_{i=n-k+1}^nx_i\Big)e^{-\frac{\beta}{2}|x|^2} \qquad\hbox{for all $x\in\Rnkp$.}
  \end{equation*}
  Although this function is not in $C^\infty_c(\mbr^{n+2k})$, it belongs to $X(\Rnkp)$. Also, both sides of the Heisenberg Uncertainty Principle (\ref{E:HUPorth2}) are finite and equal for this choice of $u$.\end{proof}

\begin{theorem}\label{T:HUPIO}
    Let $n \geq 1$, $k\in\{1,\ldots, n\}$ and $\alpha \in{ \R \backslash \{0\}}$. Then 
\begin{equation}\label{E:HUPIO}
    \begin{split}
      \alpha^2 \int_{{\Rnkp}} |\nabla u(x)|^2 \, dx  \;+ \frac{1}{\alpha^2} & \int_{{\Rnkp}} |x|^2 |u(x)|^2 dx- (n + 2k) \left( \int_{{\Rnkp}} |u(x)|^2 \, dx\right) \\  
  &=\alpha^2 \int_{\Rnkp} \Bigl| \nabla \!\Bigl(\frac{u}{\pxi} e^{\tfrac{|x|^{2}}{2\alpha^2}}\Bigr) \Bigr|^{2} w^2(x) e^{-\frac{|x|^{2}}{\alpha^2}} dx
    \end{split}
\end{equation}
for all $u \in C^\infty_c(\Rnkp)$, where 
\[w(x)=\prod_{i=n-k+1}^nx_i.\]
\end{theorem}

\begin{proof}
Since $u\in C^\infty_c(\Rnkp)$, we can write $u$ as
\[ u(x)= {w(x)} v(x),\]
where  $v  \in C^\infty_c(\Rnkp)$. 
Then,
    \begin{align*}
      &\alpha^2  |\nabla u(x)|^2 \, + \frac{1}{\alpha^2}|x|^2 \left\vert u(x) \right\vert^2  - (n + 2k) \left\vert u(x) \right\vert^2\\
      &\hskip 1in -\alpha^2  \left| \nabla \!\Bigl(\frac{u(x)}{{w(x)}} e^{\tfrac{|x|^{2}}{2\alpha^2}}\Bigr) \right|^{2} {w(x)}^2 e^{-\frac{|x|^{2}}{\alpha^2}}\\
      &=\alpha^2  \left|\nabla\left({w(x)} v(x)\right)\right|^2 \, + \frac{1}{\alpha^2} |x|^2 \left({w(x)} v(x)\right)^2 \\
      &\hskip .8in- (n + 2k) \left({w(x)} v(x)\right)^2 -\alpha^2 \Bigl| \nabla \!\Bigl( v(x) e^{\tfrac{|x|^{2}}{2\alpha^2}}\Bigr) \Bigr|^{2} {w(x)}^2 e^{-\frac{|x|^{2}}{\alpha^2}}\\
      &=\alpha^2  |\left(\nabla {w(x)}\right)v(x) + \left(\nabla v(x)\right){w(x)}|^2 \, + \frac{1}{\alpha^2} |x|^2 \left({w(x)} v(x)\right)^2 \\
      &- (n + 2k) \left({w(x)} v(x)\right)^2-\alpha^2 \Bigl| \left(\nabla v(x)\right) e^{\tfrac{|x|^{2}}{2\alpha^2}}+ \frac{x}{\alpha^2} v(x)e^{\tfrac{|x|^{2}}{2\alpha^2}}\Bigr|^{2} {w(x)}^2 e^{-\frac{|x|^{2}}{\alpha^2}}\\
      &=\alpha^2  |\bigl(\nabla{w(x)}\bigr) v(x)|^2 + \alpha^2| \bigl(\nabla v(x)\bigr){w(x)}|^2 \,\\
      &+ 2\alpha^2 \bigl(\nabla{w(x)}\bigr) v(x) \cdot \bigl(\nabla v(x)\bigr){w(x)} \, 
      + \frac{1}{\alpha^2} |x|^2 \left({w(x)} v(x)\right)^2\\
      &- (n + 2k) \left({w(x)} v(x)\right)^2-\alpha^2 \Bigl|  \bigl(\nabla v(x)\bigr) {w(x)} + \frac{x}{\alpha^2} v(x) {w(x)} \Bigr|^{2} \\
       &=\alpha^2  |\bigl(\nabla {w(x)}\bigr) v(x)|^2 + \alpha^2| \bigl(\nabla v(x)\bigr){w(x)}|^2 \, + 2\alpha^2 \bigl(\nabla({w(x)}\bigr) v(x) \cdot  \bigl(\nabla v(x)\bigr){w(x)} \, \\
      &\qquad+\frac{1}{\alpha^2} |x|^2 \left({w(x)} v(x)\right)^2 - (n + 2k) \left({w(x)} v(x)\right)^2\\
      &\qquad-\alpha^2 \Bigl|  \bigl(\nabla v(x)\bigr){w(x)} \Bigr|^{2}-  \frac{1}{\alpha^2} \Bigl|xv(x)  {w(x)} \Bigr|^{2} - 2 x\cdot \bigl(v(x) \nabla v(x)\bigr) {w(x)}^2 \\
      &= \alpha^2 \left\{  |\nabla {w(x)}|^2 v(x)^2\, +  {w(x)} \nabla {w(x)} \cdot  \bigl(\nabla v(x)^2\bigr)\right\} \, \\
      & \hskip 1in - (n + 2k) {w(x)}^2 v(x)^2 -  \bigl(x\cdot \nabla v(x)^2\bigr) {w(x)}^2.
    \end{align*}
Continuing further, we then have:
\begin{equation*}
    \begin{split}
   & \alpha^2  |\nabla u(x)|^2 \, + \frac{1}{\alpha^2}|x|^2 \left\vert u(x) \right\vert^2  - (n + 2k) \left\vert u(x) \right\vert^2 \\
      &\hskip 1in -\alpha^2  \left| \nabla \!\Bigl(\frac{u(x)}{{w(x)}} e^{\tfrac{|x|^{2}}{2\alpha^2}}\Bigr) \right|^{2} {w(x)}^2 e^{-\frac{|x|^{2}}{\alpha^2}}\\
      &= \alpha^2 \left\{  |\nabla\cdot \left( v(x)^2 {w(x)}  \nabla{w(x)}\right)-v(x)^2  {w(x)} \Delta {w(x)}\right\} \, \\
      &\qquad -{w(x)}^2 \left(n  v(x)^2 + x\cdot\bigl(\nabla  v(x)^2\bigr)\right) -2k {w(x)}^2 v(x)^2\\
      &= \alpha^2 \left\{ \nabla \left( v(x)^2 {w(x)}  \nabla{w(x)}\right)\right\} \, \\
     &\qquad -{w(x)}^2 \nabla\cdot (x v(x)^2) -2 \nabla {w(x)} \cdot x {w(x)} v(x)^2\\
      &= \alpha^2 \left\{ \nabla\cdot \left( v(x)^2 {w(x)}  \nabla{w(x)}\right)\right\} \,- \nabla\cdot \left({w(x)}^2 x v(x)^2\right) \\
       &= \alpha^2  \nabla\cdot\left\{ v(x)^2 {w(x)}  \nabla{w(x)} \,- {w(x)}^2 x v(x)^2 \right\}.\\
    \end{split}
\end{equation*}
Because $v\in C^\infty_c(\Rnkp)$, the identity (\ref{E:HUPIO}) then follows by applying the divergence theorem.
\end{proof}

The following consequence can be viewed as a scale non-invariant Heisenberg Uncertainty Principle on the orthant $\Rnkp$:

\begin{corollary}\label{C:HUPIO+}
    Let $n \geq 1$, and $k\in\{1,\ldots, n\}$. Then 
\begin{equation}\label{E:HUPIO+}
    \begin{split}
       \int_{\Rnkp} |\nabla u(x)|^2 \, dx  \;+ &\int_{\Rnkp} |x|^2 \left\vert u(x) \right\vert^2 dx- (n + 2k) \left( \int_{\Rnkp} \left\vert u(x) \right\vert^2 \, dx\right) \\  
  &= \int_{\Rnkp} \Bigl| \nabla \!\Bigl(\frac{u(x)}{\prod_{i=1}^k x_i} e^{\tfrac{|x|^{2}}{2}}\Bigr) \Bigr|^{2} e^{-|x|^{2}} \prod_{i=1}^k x_i^2\, dx
    \end{split}
\end{equation}
for all $u \in C^\infty_c({\Rnkp})$. In particular, we have the following scale non-invariant Heisenberg Uncertainty Principle
\begin{equation}\label{E:HUPIO+ineq}
       \int_{\Rnkp} |\nabla u|^2 \, dx  \;+  \int_{\Rnkp} |x|^2 \left\vert u(x) \right\vert^2 dx- (n + 2k) \left( \int_{\Rnkp} \left\vert u(x) \right\vert^2 \, dx\right) \\  
   \geq 0,
\end{equation}
for all $u \in C^\infty_c({\Rnkp})$.
\end{corollary}
\begin{proof}
   This follows from Theorem \ref{T:HUPIO}, on setting   $\alpha =1 $. 
\end{proof}

We also obtain the following ``Heisenberg Uncertainty Identity'' on the orthant $\Rnkp$:

\begin{corollary}\label{C:orthant-identity}
Let $n \geq 1$, and $k\in\{1,\ldots, n\}$. Let \[
\alpha = \left( \frac{\int_{\Rnkp} |x|^{2} |u(x)|^{2}\, dx}
{\int_{\Rnkp} |\nabla u|^{2}\, dx} \right)^{\tfrac{1}{4}}.
\]  Then 
\begin{equation}\label{eq:hupi}
    \begin{split}
        \Big(\int_{\Rnkp} |\nabla u(x)|^2 \, dx\Big)^\frac{1}{2} \; & \Big(\int_{\Rnkp} |x|^2 \left\vert u(x) \right\vert^2 dx\Big)^\frac{1}{2}- \frac{(n + 2k)}{2} \left( \int_{\Rnkp} \left\vert u(x) \right\vert^2 \, dx\right)\\  
  &= \frac{\alpha^2}{2}\int_{\Rnkp}  | \nabla \Big( \frac{u(x)}{\prod_{i=n-k+1}^n x_i} e^{\frac{|x|^2}{\alpha^2}} \Big) \Big|^2 e^{-\frac{1}{2\alpha^2}|x|^2}\prod_{i=n-k+1}^n x_i^2 dx
    \end{split}
\end{equation}
for all $u \in C_c^{\infty}({\Rnkp})$.
\end{corollary}

\begin{proof}
The result follows from the identity \eqref{E:HUPIO} in Theorem \ref{T:HUPIO} by choosing the optimal $\alpha$ on the left hand side of \eqref{E:HUPIO}.  Namely, choosing
\[
\alpha = \left( \frac{\int_{\Rnkp} |x|^{2} |u(x)|^{2}\, dx}
{\int_{\Rnkp} |\nabla u(x)|^{2}\, dx} \right)^{\tfrac{1}{4}}.
\]  
 
\end{proof}

Next, we identify the optimizers of the Heisenberg Uncertainty Principle on the orthant:

\begin{corollary}\label{C:orthant-equality}
Let $n \geq 1$, and $k\in\{1,\ldots, n\}$. Then the Heisenberg Uncertainty Principle inequality
\begin{equation}\label{HI}
\Big(\!\int_{\Rnkp} |\nabla u(x)|^2 dx\Big)^{\!1/2}
\Big(\!\int_{{\Rnkp}} |x|^2 \left\vert u(x) \right\vert^2 dx\Big)^{\!1/2}
\;\ge\; \frac{n+2k}{2}\int_{\Rnkp} \left\vert u(x) \right\vert^2 dx
\end{equation}
holds, with equality if and only if
\[
u \in \tilde{E}
:=\Big\{\Big(\prod_{i=n-k+1}^{n}x_i\Big)\,\alpha\, e^{-\beta |x|^2}
:\; \alpha\in\mathbb{R},\ \beta>0\Big\}.
\]
\end{corollary}

\begin{proof}
By Corollary~\ref{C:orthant-identity}, for fixed  
\[
\lambda = \left( \frac{\int_{\Rnkp} |x|^{2} |u(x)|^{2}\, dx}
{\int_{\Rnkp} |\nabla u|^{2}\, dx} \right)^{\tfrac{1}{4}},  
\]  
we have 
\begin{equation*}
    \begin{split}
        &\Big(\!\int_{\Rnkp} |\nabla u(x)|^2\,dx\Big)^{\!1/2}\Big(\!\int_{\Rnkp} |x|^2 \left\vert u(x) \right\vert^2\,dx\Big)^{\!1/2}-\frac{n+2k}{2}\int_{\Rnkp} \left\vert u(x) \right\vert^2\,dx\\
& \hskip 1in =\frac{\lambda^2}{2}\!\int_{\Rnkp}
\Big|\nabla\!\Big(\frac{u(x)}{\prod_{i=n-k+1}^{n}x_i}
e^{\frac{|x|^2}{2\lambda^2}}\Big)\Big|^2
e^{-\frac{|x|^2}{2\lambda^2}}\!\prod_{i=n-k+1}^{n}x_i^2\,dx \;\ge 0.
    \end{split}
\end{equation*}

Hence the equality (\ref{HI}) holds if and only the integrand vanishes a.e.; that is,
\[
\nabla\!\Big(\frac{u(x)}{\prod_{i=n-k+1}^{n}x_i}e^{\frac{|x|^2}{2\lambda^2}}\Big)=0,
\]
which gives
\(
u(x)=C\big(\prod_{i=n-k+1}^{n}x_i\big)e^{-\frac{|x|^2}{2\lambda^2}}
\)
for some \(C\in\mathbb{R}\).
Writing \(\beta=\frac{1}{2\lambda^2}>0\) yields the stated form.
Conversely, for \(u(x)=\big(\prod_{i=n-k+1}^{n}x_i\big)\lambda_0 e^{-\beta|x|^2}\) and
\(\lambda=(2\beta)^{-1/2}\), the right-hand side above is zero, so equality holds.
\end{proof}

\section{The Gaussian Poincaré Inequality with monomial weights}\label{s:GPImon}

In the following, we work with an element
\begin{equation*}
    A=(a_1,\ldots,a_n)\in {\mathbb R}_{\ge 0}^n,
    \end{equation*}
and we denote by $\left\vert A \right\vert :=\sum_{i=1}^{n}a_i$, and by $\RnAp$ the following subset of $\R^n$:
\begin{equation}\label{E:RnAp2}
    \RnAp:=\{x\in\R^n\,:\, \hbox{$x_i>0$ if $a_i>0$}\}.
\end{equation}
(There is a slight abuse of notation from (\ref{E:RnAp}), but the relationship between (\ref{E:RnAp}) and (\ref{E:RnAp2}) is clear.)
Let $\mu_A$ and, for $\lambda>0$, $\mu_{A,\lambda}$, be the measures given by:
\begin{equation}\label{E:defmuAlam}
\begin{split}
    d\mu_A(x)&= \frac{x^A e^{-\frac{1}{2}|x|^2}\,dx}{\int_{\RnAp}x^A e^{-\frac{1}{2}|x|^2} dx}\\
    d\mu_{\lambda,A}(x)
&=\frac{x^A e^{-\frac{|x|^2}{2\lambda^2}}\,dx}
{\int_{\RnAp}x^A e^{-\frac{|x|^2}{2\lambda^2}}dx},
\end{split}
\end{equation} 
where
\begin{equation}\label{E:xA}
    x^A= \prod_{i=1}^{n} x^{a_i}.
\end{equation}
In the following, we denote by $\mathcal{X}_A$ the closure of $C^\infty_c(\RnAp)$ in the norm
\begin{equation*}
    \|u\|_A:=\left (\int_{\RnAp} |u|^2\,d\mu_A +\int_{\RnAp}|\nabla u|^2\,d\mu_A \right )^{\frac{1}{2}},
\end{equation*}
and by $\mathcal{X}_{A,\lambda}$ the closure of $C^\infty_c(\RnAp)$ in the norm
\begin{equation*}
    \|u\|_{A,\lambda}:=\left (\int_{\RnAp} |u|^2\,d\mu_{A,\lambda} +\int_{\RnAp}|\nabla u|^2\,d\mu_{A,\lambda} \right )^{\frac{1}{2}}.
\end{equation*}

The \emph{classical Gaussian Poincaré inequality} states that for every smooth function 
\(u:\mathbb{R}^N \to \mathbb{R}\) and the normalized Gaussian measure
\[
d\mu(x) = \frac{e^{-\frac{|x|^2}{2}}}{\int_{\mathbb{R}^N} e^{-\frac{|x|^2}{2}} dx} \, dx,
\]
we have
\[
\int_{\mathbb{R}^N} |\nabla u|^2 \, d\mu
\geq
\int_{\mathbb{R}^N} \left| u - \int_{\mathbb{R}^N} u \, d\mu \right|^2 d\mu.
\]
The equality holds \emph{if and only if}
\[
u(x) = \mathbf{a} \cdot x + b, \qquad \mathbf{a} \in \mathbb{R}^N, \quad b \in \mathbb{R}.
\]
That is, the optimizers are precisely the affine linear functions, corresponding to the first non-constant eigenspace of the Ornstein-Uhlenbeck operator. The proof of the Gaussian Poincaré inequality follows from the spectral analysis of the Ornstein-Uhlenbeck operator; see \cite{LLR25}, for instance. 
Moreover, the linearization of the logarithmic Sobolev inequality with respect to the Gaussian measure established by Leonard Gross in \cite{Gro75} also implies the Gaussian Poincaré inequality.

A \textit{scaled version} of the Gaussian Poincaré inequality holds for the measure
\[
d\mu_\lambda(x)
=\frac{e^{-\frac{|x|^2}{2\lambda^2}}}{\int_{\mathbb{R}^N}e^{-\frac{|x|^2}{2\lambda^2}}dx}\,dx,
\qquad \lambda>0,
\]
in which case
\begin{equation}\label{E:poincare}
    \int_{\mathbb{R}^N}|\nabla u|^2\,d\mu_\lambda
\;\ge\;
\frac{1}{\lambda^2}
\int_{\mathbb{R}^N}\!\Big|u-\!\int_{\mathbb{R}^N}\!u\,d\mu_\lambda\Big|^2 d\mu_\lambda,
\end{equation}
and the equality again holds if and only if \(u(x)=a\cdot x+b\).
The scaling parameter \(\lambda\) simply rescales the measure and does not change the optimizer structure. The following results were proved by Lam, Lu and Russanov \cite{LLR25}.

\begin{theorem}[{\cite{LLR25}}]\label{T:gaussian-poincare-weighted}
For all \(u\in \mathcal{X}_A\), we have
\begin{equation}\label{GaussianPoincareMonomial}
\int_{\RnAp}|\nabla u|^2\,d\mu_A
\;\ge\;
\int_{\RnAp}\!\Big|u-\!\int_{\RnAp}\!u\,d\mu_A\Big|^2 d\mu_A.
\end{equation}
Moreover, if \(x^A\) is partial, that is when there exists some $i$ such that $a_i=0$, then the equality in \eqref{GaussianPoincareMonomial} can be attained by non-constant functions.
\end{theorem}

\begin{theorem}[{\cite{LLR25}}]\label{thm:gaussian-poincare-lambda}
For \(\lambda>0\) and \(u\in \mathcal{X}_{\emptyset,\lambda}\),
\[
\int_{\mathbb{R}^N}|\nabla u|^2\,d\mu_\lambda
\;\ge\;
\frac{1}{|\lambda|^2}
\inf_{c\in \R,\mathbf{d}\in \R^N}
\int_{\mathbb{R}^N}
\Big(|u-c|^2 + |u-c-\mathbf{d}\cdot x|^2\Big)
d\mu_\lambda,
\]
and the equality can be attained by non-linear functions.
\end{theorem}

In this work, we present an alternative proof of Theorem~\ref{T:gaussian-poincare-weighted} that holds for 
\textit{integer weights} \(A\in\mathbb{Z}^N\), and subsequently extend this approach to establish 
Theorem~\ref{thm:gaussian-poincare-lambda} for the same integer-weighted case \(x^A\).

\begin{theorem}[Scaled Poincar\'e inequality with Gaussian monomial weights]\label{SPoin}
Let $A=(a_1,\ldots,a_n)\in {\mathbb{Z}}_{\geq0}^n$ and $\lambda>0$. 
Then for all $u \; \in \mathcal{X}_{A,\lambda}$, we have 
\[
\int_{\RnAp}|\nabla u|^2\, d\mu_{\lambda,A} 
\;\geq\;
\frac{1}{\lambda^2}
\int_{\RnAp} 
\Big(u(x)-\int_{\RnAp} u(x)\, d\mu_{\lambda,A}\Big)^2 d\mu_{\lambda,A}(x),
\]
where 
\[
d\mu_{\lambda,A}(x)
=\frac{x^A e^{-\frac{|x|^2}{2\lambda^2}}}
{\int_{\RnAp}x^A e^{-\frac{|x|^2}{2\lambda^2}}dx},
\qquad
x^A=\prod_{i=1}^n x_i^{a_i}.
\]

\end{theorem}

\begin{proof}
Let $v \in C^\infty_c(\RnAp)$.  
We lift $v$ to a function on 
$\R^{a_1+1} \times \cdots \times \R^{a_n+1}$  
by setting
\begin{equation}\label{E:lifttlvA}
\tlv(y_1,\dots,y_n)=v(x_1,\dots,x_n),
\qquad
x_i=|y_i|\ \text{ if } a_i>0,\ \text{and } x_i=y_i\ \text{ if } a_i=0,
\end{equation}
where each $y_i$ runs over $\R^{a_i+1}$.  
Since $v$ is $C^\infty_c(\RnAp)$ and vanishes near the boundary,  
we have $\tlv\in C^\infty_c(\R^{n+|A|})$. Note also that 
$$|x|^2=|y|^2.$$

Then, the classical scaled Gaussian Poincar\'e inequality gives
\[
\int_{\mathbb{R}^{n+|A|}} |\nabla \tlv(y)|^2 d\mu_\lambda(y)
\;\ge\;
\frac{1}{\lambda^2}
\int_{\mathbb{R}^{n+|A|}} 
\Big|
\tlv(y) - 
\int_{\mathbb{R}^{n+|A|}} \tlv(y) d\mu_\lambda(y)
\Big|^2
d\mu_\lambda(y),
\]
where 
\[
d\mu_\lambda(y)=
\frac{e^{-\frac{|y|^2}{2\lambda^2}}}
{\int_{\R^{n+|A|}} e^{-\frac{|y|^2}{2\lambda^2}} dy}\,dy.
\]

As before, for each $i\in A$, for the block $y_i=(y_{i,1},\dots,y_{i,a_i})$, we have
\[
\frac{\partial \tlv}{\partial y_{i,j}}
= \frac{\partial v}{\partial r_i}(x',r_1,\dots,r_k)\frac{y_{i,j}}{r_i},
\qquad r_i=|y_i|,
\]
and therefore
\[
|\nabla_{y_i}\tlv|^2
=(\partial_i v)^2.
\]
Combining all derivatives, we again obtain
\begin{equation}\label{E:nablatlvsqdv}
|\nabla\tlv(y)|^2
=|\nabla v(x',|y_1|,\dots,|y_k|)|^2.
\end{equation}
Thus, integrating over $\R^{n+|A|}$ in spherical coordinates yields, on using  the identity (\ref{E:inttlvx'yA}) (for the function $|\nabla v(x)|e^{-\frac{|x|^2}{4\lambda^2}}$ in place of $v(x)$), we have:
\begin{equation*}
\begin{split}
\int_{\R^{n+|A|}} |\nabla \tlv(y)|^2 d\mu_\lambda(y)
&=\prod_{i=1}^{n} |\mathbb{S}^{a_i}|
\int_{\RnAp} |\nabla v(x)|^2 
\frac{x^A e^{-\frac{|x|^2}{2\lambda^2}}}
{\prod_{i=1}^{n} |\mathbb{S}^{a_i}|\int_{\RnAp}x^A e^{-\frac{|x|^2}{2\lambda^2}}dx}\,dx\\
&=\int_{\RnAp} |\nabla v(x)|^2\, d\mu_{\lambda,A}(x).
\end{split}
\end{equation*}

Similarly, for the right-hand side, we have
\begin{equation*}
\begin{split}
\int_{\R^{n+|A|}} 
\Big|\tlv(y)-\!\int_{\R^{n+|A|}}\!\tlv(y)\,d\mu_\lambda(y)\Big|^2 d\mu_\lambda(y)
&=\int_{\RnAp} 
\Big|v(x)-\!\int_{\RnAp}\!v \,d\mu_{\lambda,A}\Big|^2
d\mu_{\lambda,A}(x).
\end{split}
\end{equation*}
Combining both expressions gives the desired inequality.
\end{proof}

\begin{corollary}[Equality in the scaled Gaussian Poincaré inequality with monomial weights]\label{cor:eq-Poincare-weighted}
Let $A=(a_1,\dots,a_n)\in\mathbb{Z}_{\ge 0}^n$ and $\lambda>0$.  
With
\[
d\mu_{A,\lambda}(x)
=\frac{x^A\,e^{-\frac{|x|^2}{2\lambda^2}}}{\displaystyle\int_{\RnAp} x^A\,e^{-\frac{|x|^2}{2\lambda^2}}\,dx}\,dx,
\qquad
x^A:=\prod_{i=1}^n x_i^{a_i},
\]
the inequality
\[
\int_{\RnAp}|\nabla v|^2\,d\mu_{A,\lambda}
\;\ge\;
\frac{1}{\lambda^2}\int_{\RnAp}\Big|\,v(x)-\!\!\int_{\RnAp}\! v\,d\mu_{A,\lambda}\Big|^2 d\mu_{A,\lambda}(x)
\]
holds for all $v\in X_{A,\lambda}$.  
Moreover, {equality holds if and only if}
\[
v(x)=\sum_{i:\,a_i=0}d_i\,x_i+c, 
\qquad d_i,c\in\R,
\]
i.e., the linear part of $v$ involves only the coordinates with zero weight $a_i=0$.  
(The same characterization applies to the unscaled case $\lambda=1$.)
\end{corollary}

\begin{proof}
Let $A=(a_1,\dots,a_n)\in\mathbb{Z}_{\ge0}^n$ and let
\[
d\mu_{A,\lambda}(x)=\frac{x^A e^{-\frac{|x|^2}{2\lambda^2}}}{\int_{\RnAp} x^A e^{-\frac{|x|^2}{2\lambda^2}}dx}\,dx.
\]
As in (\ref{E:lifttlvA}), we use the lift $\tilde v:\R^n\to\R$, given  by
\[
\tilde v(y):=v(x),\qquad
x_i=
\begin{cases}
|y_i|,& \hbox{if $ a_i>0$},\\[2pt]
y_i,& \hbox{if $a_i=0$}.
\end{cases}
\]
Then $\tilde v$ is even in each coordinate with $a_i>0$, and the change of variables shows that the scaled
Poincar\'e inequality for $v$ with respect to $\mu_{A,\lambda}$ is equivalent to the classical Gaussian one for
$\tilde v$ with respect to the Gaussian measure $\mu_\lambda$.

For the classical Gaussian inequality, equality holds if and only if $\tilde v(y)=\alpha\cdot y+\beta$ with
$\alpha\in\R^n$, $\beta\in\R$. Since $\tilde v$ is even in $y_i$ for every $a_i>0$, the odd term
$\alpha_i y_i$ must vanish; hence $\alpha_i=0$ whenever $a_i>0$. Writing the result back in $x$–variables,
we obtain
\[
v(x)=\sum_{i:\,a_i=0}\alpha_i x_i+\beta.
\]

Conversely, if $v(x)=\sum_{i:\,a_i=0}\alpha_i x_i+\beta$, then the lift is
$\tilde v(y)=\alpha\cdot y+\beta$ with $\alpha_i=0$ for $a_i>0$, which attains equality in the classical
Gaussian Poincaré inequality, hence also in the weighted/scaled inequality after pushing down.
This proves the characterization.
\end{proof}

\begin{theorem}\label{T:weighted-poincare-stability}
For $\lambda>0$ and $v\in X_{A,\lambda,0}$, we have
\begin{equation*}
\begin{split}
\int_{\RnAp} |\nabla v|^2\, d\mu_{\lambda,A}-\frac{1}{|\lambda|^2}
\inf_{c\in \R}
\Big(\int_{\RnAp}|v-c|^2 \; d\mu_{\lambda,A}\Big)&\\
&\hskip -2in 
\;\ge\;
\frac{1}{|\lambda|^2}
\inf_{c,d_i\in \R}
\int_{\RnAp}
\Big(
 |v(x)-c-\sum_{i:a_i=0}(d_i\cdot x_i)|^2
\Big)
d\mu_{\lambda,A}(x),
\end{split}
\end{equation*}
where
\[
d\mu_{\lambda,A}(x)
=
\frac{x^A\, e^{-\frac{|x|^2}{2\lambda^2}}}
     {\displaystyle\int_{\RnAp} x^A\, e^{-\frac{|x|^2}{2\lambda^2}}\,dx}\,dx,
\qquad
x^A = \prod_{i=1}^n x_i^{A_i}.
\]
\end{theorem}

\begin{proof}

Let $v \in C^\infty_c(\RnAp)$.  
We lift $v$ to a function on 
$\R^{a_1+1} \times \cdots \times \R^{a_n+1}$  
by setting
\[
\tlv(y_1,\dots,y_n)=v(x_1,\dots,x_n),
\qquad
x_i=|y_i|\ \text{ if } a_i>0,\ \text{and } x_i=y_i\ \text{ if } a_i=0,
\]
where each $y_i$ runs over $\R^{a_i+1}$.  
Since $v$ is $C^\infty_c(\RnAp)$ and vanishes near the boundary,  
we have   \(\tlv \in X_{\lambda,0}\), then for \(\lambda>0\)
\begin{equation}\label{E:stability-lift-poincare}
    \begin{split}
        \int_{\mathbb{R}^{n+|A|}}|\nabla \tlv|^2\,d\mu_\lambda
    \;&\ge\;\frac{1}{|\lambda|^2} \inf_{c\in \R,\mathbf{d}\in \R^{n+|A|}} \int_{\mathbb{R}^{n+|A|}} \Big(|\tlv(x)-c|^2 + |\tlv-c-\mathbf{d}\cdot x|^2\Big) d\mu_\lambda(x),\\
    &\geq \frac{1}{|\lambda|^2} \inf_{c\in \R} \Big(\int_{\mathbb{R}^{n+|A|}} |\tlv-c|^2 d\mu_\lambda \Big)\\
    & \hskip 1in + \frac{1}{|\lambda|^2} \inf_{c\in \R,\mathbf{d}\in \R^{n+|A|}} \Big(\int_{\mathbb{R}^{n+|A|}} |\tlv(x)-c-\mathbf{d}\cdot x|^2\Big) d\mu_\lambda(x),\\
    \end{split}
\end{equation}

where 
\[
d\mu_\lambda(y)=
\frac{e^{-\frac{|y|^2}{2\lambda^2}}}
{\int_{\R^{n+|A|}} e^{-\frac{|y|^2}{2\lambda^2}} dy}\,dy.
\]

As before, each  $y_i=(y_{i,1},\dots,y_{i,a_i})$ and $d_i=(d_{i,1},\dots,d_{i,a_i})$

\begin{equation*}
\begin{split}
\int_{\R^{n+|A|}} |\nabla \tlv(y)|^2 d\mu_\lambda(y)
&=\int_{\RnAp} |\nabla v(x)|^2\, d\mu_{\lambda,A}(x).
\end{split}
\end{equation*}


Then
\[
\int |\tilde v(y)-c-\mathbf{d}\cdot y|^2\,d\mu =\int |\tlv-c|^2\,d\mu
-2\sum_{i=1, j=1}^{n,a_i} d_{i,j}\int (\tlv(y)-c)  y_{i,j} d\mu_\lambda(y)
+\sum_{i=1, j=1}^{n,a_i}  d_{i,j}^2 \int y_{i,j}^2\,d\mu(y),
\]
because $\int y_{i,j} y_{l,m}\,d\mu=0$ for ${i,j}\ne (l,m)$ by evenness of the measure.

For each coordinate the quadratic in $d_{i,j}$ is minimized at
\[
d^*_{i,j}=\frac{\int (\tlv(y)-c)  y_{i,j} d\mu_\lambda(y)}{\int y_{i,j}^2\,d\mu(y)},
\]
\[\inf_{d_{i,j}}\big(d_{i,j}^2\int y_{i,j}^2\,d\mu(y) -2d_{i,j}\int (\tlv(y)-c)  y_{i,j} d\mu_\lambda(y)\big)=-\,\frac{(\int (\tlv-c)  y_{i,j} d\mu_\lambda(y))^2}{\int y_{i,j}^2\,d\mu(y)}.
\]
Hence
\[
\frac{1}{|\lambda|^2} \inf_{c\in \R,\mathbf{d}\in \R^{n+|A|}} \int_{\mathbb{R}^{n+|A|}} |\tlv(y)-c-\mathbf{d}\cdot x|^2 d\mu_\lambda
=\int |\tilde v-c|^2\,d\mu-\sum_{i=1, j=1}^{n,a_i} \,\frac{(\int (\tlv(y)-c)  y_{i,j} d\mu_\lambda(y))^2}{\int y_{i,j}^2\,d\mu(y)}.
\]

 If $a_i>0$, then $\tilde v$ is even in $y_i$, so $t$ is even in $y_i$, and
$y_i$ is odd; therefore 
$$\int (\tlv-c)  y_{i,j} d\mu_\lambda(y)=0.$$
Thus, $d_{i,j}^*=0$ whenever $a_i>0$.

\medskip
 Identifying $y_i$ with $x_i$ for the coordinates with $a_i=0$ gives
\[
\frac{1}{|\lambda|^2} \inf_{c\in \R,\mathbf{d}\in \R^{n+|A|}} \int_{\mathbb{R}^{n+|A|}} |\tlv(x)-c-\mathbf{d}\cdot x|^2 d\mu_\lambda
=\frac{1}{|\lambda|^2} \inf_{\substack{c\in \R,\mathbf{d} \in\mathbb{R}^n:\\ d_i=0\ \text{if }a_i>0}}
\int |v(x)-c-\mathbf{d}\cdot x|^2\,d\mu_{\lambda,A}(x).
\]

\end{proof}

\section{Stability of the Heisenberg Uncertainty Principle on an Orthant-Proof of Theorem \ref{T2}}\label{s:StabHUPOrth}
The main purpose of this section is to establish the stability of the Heisenberg Uncertainty Principle on an orthant. Furthermore, we will also prove the stability of the stability inequality for the Heisenberg Uncertainty Principle on an orthant.

Recall that $S({\Rnkp})$ is the completion of $C_c^\infty({\Rnkp})$ under the norm 
\begin{equation*}
    \left (\int_{{\Rnkp}} |\nabla u(x)|^2 \, dx \right )^{\frac{1}{2}} +  \left (\int_{{\Rnkp}} |x|^2 |u(x)|^2 dx\right )^{\frac{1}{2}}.
\end{equation*}
We first establish the stability for the scale non-invariant Heisenberg Uncertainty Principle on an othant. 
\begin{theorem}\label{T:staborthHeis}
For any $u \in S({\Rnkp})$, we have
\begin{equation*}
    \begin{split}
       \int_{{\Rnkp}} |\nabla u(x)|^{2} \, dx + \int_{{\Rnkp}} |x|^{2} |u(x)|^{2} \, dx &- (n+2k) \int_{{\Rnkp}} |u(x)|^{2} \, dx\\
&\ge 2 \inf_{c} \int_{{\Rnkp}} \left| u(x) - c \prod^n_{i=n-k+1}x_i^2 e^{-\frac12 |x|^{2}} \right|^{2} dx . 
    \end{split}
\end{equation*}
\emph{The constant $2$ on the right hand side is sharp and the equality holds 
if and only if $u$ is of the form
\[
u(x)
= \Bigg(\prod_{i=n-k+1}^{n} x_i\Bigg)
\, e^{-\frac{|x|^2}{2}}
\Bigg(\sum_{i=1}^{n-k} b_i x_i + b
\Bigg),
\qquad b_1,\dots,b_{n-k},b\in\mathbb{R}.
\]}
\end{theorem}

\begin{proof}
By a density argument, we can assume that $u \in C^\infty_c(\mbr^n_{k,+})$. By Corollary \ref{C:HUPIO+}, we have
    \begin{equation}\label{HUPIO+}
    \begin{split}
       \int_{\Rnkp} |\nabla u(x)|^2 \, dx  \;+ &\int_{{\Rnkp}} |x|^2 \left\vert u(x) \right\vert^2 dx- (n + 2k) \left( \int_{\Rnkp} u^2 \, dx\right) \\  
  &= \int_{\Rnkp} \Bigl| \nabla \!\Bigl(\frac{u}{\prod_{i=1}^k x_i} e^{\tfrac{|x|^{2}}{2}}\Bigr) \Bigr|^{2} e^{-|x|^{2}} \prod_{i=1}^k x_i^2\, dx.
    \end{split}
\end{equation}
Using the Poincar\'{e} inequality (Theorem \ref{SPoin}) we have:
 \begin{equation*}
    \begin{split}
       \int_{\Rnkp} |\nabla u(x)|^2 \, dx  \;+ &\int_{\Rnkp} |x|^2 \left\vert u(x) \right\vert^2 dx- (n + 2k) \left( \int_{\Rnkp} \left\vert u(x) \right\vert^2 \, dx\right) \\  
  &\geq 2\inf_{C\in \R}\int_{\Rnkp} \left|\left(\frac{u(x)}{\prod_{i=1}^k x_i} e^{\tfrac{|x|^{2}}{2}}\right)-C \right|^{2} e^{-|x|^{2}} \prod_{i=1}^k x_i^2\, dx\\
  &=2\inf_{C\in \R}\int_{\Rnkp} \Bigl|\ u(x)-Ce^{-\frac{|x|^{2}}{2}} \prod_{i=1}^k x_i \Bigr|^{2} \, dx.\\
    \end{split}
\end{equation*}

Finally, using Corollary \ref{C:HUPIO+} and Corollary \ref{cor:eq-Poincare-weighted} with $\lambda=1$, we have that the equality holds if and only if $v(x)= \sum_{i:a_i=0} b_i x_i+b \; \text{ for } b_i,b\in \R$ with $$v(x)=\frac{u(x)}{\prod_{i=n-k+1}^n x_i} e^{\tfrac{|x|^{2}}{2}}.$$
Equivalently, $$ u(x)
= \Bigg(\prod_{i=n-k+1}^{n} x_i\Bigg)
\, e^{-\frac{|x|^2}{2}}
\Bigg(\sum_{i=1}^{n-k} b_i x_i + b
\Bigg),
\qquad b_1,\dots,b_{n-k},b\in\mathbb{R}. $$
\end{proof}

We are ready now to prove the stability for the scale-invariant Heisenberg Uncertainty Principle on an orthant. 
\begin{proof}[Proof of Theorem \ref{T2}.]
    Let $u \in C^\infty_c(\Rnkp)$, then by \eqref{eq:hupi} we have.
    \begin{equation*}
    \begin{split}
       \Big( \int_{\Rnkp} |\nabla u(x)|^2 \, dx\Big)^\frac{1}{2}  \; & \left(\int_{{\Rnkp}} |x|^2 |u(x)|^2 dx\right)^\frac{1}{2}- \frac{(n + 2k)}{2}  \int_{{\Rnkp}} |u(x)|^2 \, dx \\  
      &= \frac{\lambda^2}{2}\int_{\Rnkp}  \left| \nabla \Big( \frac{u(x)}{\prod_{i=n-k+1}^n x_i} e^{\frac{|x|^2}{2\lambda^2}} \Big) \right|^2 e^{-\frac{1}{\lambda^2}|x|^2}\prod_{i=n-k+1}^n x_i^2 dx.
    \end{split}
\end{equation*}
   
Using Theorem \ref{SPoin}, we have:
 \begin{align*}
       \Big( \int_{{\Rnkp}} |\nabla u(x)|^2 \, dx \Big)^\frac{1}{2} \; & \left(\int_{{\Rnkp}} |x|^2 |u(x)|^2 dx \right)^\frac{1}{2}- \frac{(n + 2k)}{2}  \int_{{\Rnkp}} |u(x)|^2 \, dx  \\  
  &\geq \inf_{C\in \R, \lambda \neq 0}\int_{\Rnkp} \left|\!\left(\frac{u(x)}{\prod_{i=1}^k x_i} e^{\tfrac{|x|^{2}}{2\lambda^2}}\right)-C \right|^{2} e^{-\frac{|x|^{2}}{\lambda^2}} \prod_{i=1}^k x_i^2\, dx\\
  &=\inf_{C\in \R, \lambda \neq 0}\int_{\Rnkp} \Bigl|\ u(x)-Ce^{-\frac{|x|^{2}}{2\lambda^2}} \prod_{i=1}^k x_i \Bigr|^{2} \, dx\\
  &=\inf_{v\in \tilde{E}}\int_{\Rnkp} \Bigl|\ u(x)-v(x)|^{2} \, dx.
    \end{align*}
Now, using Corollary \ref{C:orthant-identity} and Corollary \ref{cor:eq-Poincare-weighted}, we have that the equality holds if and only if $v(x)= \sum_{i:a_i=0} b_i x_i+b \; \text{ for } b_i,b\in \R$ with $$v(x)=\frac{u(x)}{\prod_{i=n-k+1}^n x_i} e^{\tfrac{|x|^{2}}{2\lambda^2}}.$$
Equivalently, $$ u(x)
= \Bigg(\prod_{i=n-k+1}^{n} x_i\Bigg)
\, e^{-\frac{|x|^2}{2\lambda^2}}
\Bigg(\sum_{i=1}^{n-k} b_i x_i + b
\Bigg),
\qquad b_1,\dots,b_{n-k},b\in\mathbb{R}. $$
\end{proof}

\begin{theorem}
For all $u \in S({\Rnkp})$:
\[
\rho_1(u) \ge \frac12 \inf_{\omega \in \tilde{E}} 
\left\{ \int_{{\Rnkp}} |u(x) - \omega(x)|^2 \, dx 
\; : \; \int_{{\Rnkp}} |u(x)|^2 \, dx = \int_{{\Rnkp}} |\omega(x)|^2 \, dx \right\}.
\]
\emph{Moreover, the inequality is sharp and the equality can be attained by nontrivial functions.}
\end{theorem}

\begin{proof}
    Let $u \in C_c^\infty({\Rnkp})$.  Then there exists ${v} \in  C_c^\infty({\Rnkp})$ such that \begin{equation*}
        u(x)=\left(\prod_{i=n-k+1}^{n} x_i\right)    v(x)\qquad\hbox{  for all $x\in\Rnkp$.}
    \end{equation*}
    We lift $v$ to a function $\tlv$ on $\R^{n-k}\times (\R^3)^k$  given by
    \begin{equation*}
        \tlv(x',y)=v(x', |y_1|,\ldots, |y_k|)
    \end{equation*}
    for all $x'\in\R^{n-k}$ and with each $y_i$ running over $\R^3$. Note that since $v$ is smooth, of compact support, and vanishes near the boundary of $\Rnkp$, we have $\tlv\in C^\infty_c\bigl(\R^{n-k}\times (\R^3)^k\bigr)$.
    Then By \cite{CFLL24} we have
    \[
    \delta_1(\tlv) \ge \frac12 \inf_{\tilde{\omega} \in {E}} 
    \left\{ \int_{\mathbb{R}^{n-k} \times \R^{3k}} |\tlv - \tilde{\omega}|^2 \, dx dy 
    \; : \; \int_{\mathbb{R}^{n-k} \times \R^{3k}} |\tlv|^2 \, dx dy = \int_{\mathbb{R}^{n-k} \times \R^{3k}} |\tilde{\omega}|^2 \, dx dy \right\}
    \]
    where, \[
\delta_{1}(\tlv)
= 
\left( \int_{\mathbb{R}^{n-k} \times \R^{3k}} |\nabla \tlv|^{2}\, dxdy \right)^{\frac{1}{2}}
\left( \int_{\mathbb{R}^{n-k} \times \R^{3k}} |x|^{2} |\tlv|^{2}\, dx dy\right)^{\frac{1}{2}}
- \frac{n+2k}{2} \int_{\mathbb{R}^{n-k} \times \R^{3k}} |u|^{2}\, dxdy.
\]
Then by Theorem \ref{T:integration}

\begin{equation*}
    \begin{split}
\rho_1(u) \ge \frac12 \inf_{\omega \in {E}} 
\left\{ \int_{{\Rnkp}} |u - \pxi \omega|^2 \, dx 
\; : \; \int_{{\Rnkp}} |u|^2 \, dx = \int_{{\Rnkp}} |\pxi\omega|^2 \, dx \right\}
    \end{split}
\end{equation*}
where $\omega= \alpha e^{-\beta|x|^2}$ if $ \tilde{\omega}= \alpha e^{-\beta(|x|^2|y|^2)}$. Therefore

\begin{equation*}
    \begin{split}
\rho_1(u) \ge \frac12 \inf_{\omega \in \widetilde {E}} 
\left\{ \int_{{\Rnkp}} |u(x) -  \omega(x)|^2 \, dx 
\; : \; \int_{{\Rnkp}} |u(x)|^2 \, dx = \int_{{\Rnkp}} |\omega(x)|^2 \, dx \right\}.
    \end{split}
\end{equation*}

Let $u(x)=x_1(\prod_{i=n-k+1}^nx_i)e^{-\frac{|x|^2}{2}}$,
we get

 $$\Big(\int_{\Rnkp}|\nabla u(x) |^2 dx\Big)^\frac{1}{2}\Big( \int_{\Rnkp} |x|^2 |u(x)|^2\ dx\Big)^\frac{1}{2}- \frac{(n+2k)}{2}\int_{\Rnkp}|u(x)|^2 \ dx = \frac{1}{2|\mathbb{S}^2|^k}\pi^{\frac{n+2k}{2}}.$$
 Also,
$$|\mathbb{S}^2|^k \int_{{\Rnkp}} \left| c \prod^n_{i=n-k+1}x_i^2 e^{-\frac{1}{2\lambda^2} |x|^{2}} \right|^{2} dx =  c^{2}\lambda^{n+2k}\pi^{\frac{n+2k}{2}} $$
$$|\mathbb{S}^2|^k \int_{{\Rnkp}} \left| u(x) \right|^{2} dx= \frac{\pi^{\frac{n+2k}{2}}}{2}.$$
Therefore
\begin{equation*}
    \begin{split}
        & \inf_{\omega \in \tilde{E}} 
\left\{ \int_{{\Rnkp}} |u(x) - \omega(x)|^2 \, dx 
\; : \; \int_{{\Rnkp}} |u(x)|^2 \, dx = \int_{{\Rnkp}} |\omega(x)|^2 \, dx \right\} \\
&=\inf_{c,\lambda}\left\{ \frac{\pi^{\frac{n+2k}{2}}}{2|\mathbb{S}^2|^k}+ \frac{c^{2}\lambda^{n+2k}\pi^{\frac{n+2k}{2}}}{|\mathbb{S}^2|^k}
\; : \;  \frac{\pi^{\frac{n+2k}{2}}}{2|\mathbb{S}^2|^k}= \frac{c^{2}\lambda^{n+2k}\pi^{\frac{n+2k}{2}}}{|\mathbb{S}^2|^k} \right\}
= \frac{\pi^{\frac{n+2k}{2}}}{|\mathbb{S}^2|^k}.
    \end{split}
\end{equation*}

\end{proof}

\begin{theorem}
For all $u \in C^\infty_c(\Rnkp)$, then
\begin{align*}
    & \int_{{\Rnkp}} |\nabla u(x)|^2 \, dx    + \int_{{\Rnkp}} |x|^2 |u(x)|^2 \, dx - (n+2k) \int_{{\Rnkp}} |u(x)|^2 \, dx \\
& \ge \frac{2}{n+2k+3} \inf_{c \in \mathbb{R}} \left(
\int_{{\Rnkp}} \left| \nabla \left(u(x) - c\prod^n_{i=n-k+1}x_ie^{-\frac12 |x|^2} \right) \right|^2  dx + \right.  \\ &\qquad\left.\int_{{\Rnkp}} |x|^2 \left| u(x) - c \prod^n_{i=n-k+1}x_ie^{-\frac12 |x|^2} \right|^2 dx + \int_{{\Rnkp}} \left| u(x) - c \prod^n_{i=n-k+1}x_i e^{-\frac12 |x|^2} \right|^2 dx \right ).
   \end{align*}
\emph{The inequality is sharp and the equality can be attained by nontrivial functions.}
\end{theorem}

\begin{proof}

 Let $u \in C_c^\infty({\Rnkp})$.  Then there exists ${v} \in  C_c^\infty({\Rnkp})$ such that \begin{equation*}
        u(x)=\left(\prod_{i=n-k+1}^{n} x_i\right)    v(x)\qquad\hbox{  for all $x\in\Rnkp$.}
    \end{equation*}
    We lift $v$ to a function on $\R^{n-k}\times (\R^3)^k$  given by
    \begin{equation*}
        \tlv(x',y)=v(x', |y_1|,\ldots, |y_k|)
    \end{equation*}
    for all $x'\in\R^{n-k}$ and with each $y_i$ running over $\R^3$. Note that since $v$ is smooth, of compact support, and vanishes near the boundary of $\Rnkp$, we have $\tlv\in C^\infty_c\bigl(\R^{n-k}\times (\R^3)^k\bigr)$.
    Then by \cite{CFLL24} we have
    \begin{equation*}
        \begin{split}  
    \int_{\mathbb{R}^{n+2k}}& |\nabla \tilde{v}(x',y)|^{2}\,dx'\,dy
    + \int_{\mathbb{R}^{n+2k}} |x|^{2} |\tilde{v}(x',y)|^{2}\,dx'\,dy
    - (n+2k) \int_{\mathbb{R}^{n+2k}} |\tilde{v}(x',y)|^{2}\,dx'\,dy\\
   & \ge \frac{2}{n+2k+3}\inf_{c \in \mathbb{R}} \Big( \int_{\mathbb{R}^{n+2k}} \left| \nabla\!\left(\tilde{v}(x',y)
    - c e^{-\frac{1}{2}|x|^{2}}\right) \right|^{2}\,dx'\,dy \\
    &+ \int_{\mathbb{R}^{n+2k}} |x|^{2} 
        \left| \tilde{v}(x',y) - c e^{-\frac{1}{2}|x|^{2}} \right|^{2}\,dx'\,dy
        + \int_{\mathbb{R}^{n+2k}} 
        \left| \tilde{v}(x',y) - c e^{-\frac{1}{2}|x|^{2}} \right|^{2}\,dx'\,dy\Big).
        \end{split}
    \end{equation*}
    Let 
    $$\tilde{w} (x',y)= \tilde{v}(x',y) - c e^{-\frac{1}{2}(|x'|^{2}+|y|^2)}$$
    and 
    $$w(x) =v(x)-c e^{-\frac{1}{2}|x|^2}.$$
    By Theorem \ref{T:integration},
    \begin{align*}  
    \int_{\Rnkp}& |\nabla u(x)|^{2}\,dx 
    + \int_{\Rnkp} |x|^{2} |u(x)|^{2}\,dx 
    - (n+2k) \int_{\Rnkp} |u(x)|^{2}\,dx \\
   & \ge \frac{2}{n+2k+3}\inf_{c \in \mathbb{R}} \left( \int_{\mathbb{R}^{n+2k}} | \nabla\tilde{w}(x',y)|^{2}\,dx'\,dy + \int_{\mathbb{R}^{n+2k}} |x|^{2} 
        | \tilde{w}(x',y)|^{2}\,dx'\,dy\right.\\
        &\left. \qquad
        + \int_{\mathbb{R}^{n+2k}} 
        \left| \tilde{w}(x',y)\right|^{2}\,dx'\,dy\right)\\
        & = \frac{2}{n+2k+3}\inf_{c \in \mathbb{R}} \left( \int_{\Rnkp} \left| \nabla \left(\pxi w(x) \right) \right|^{2}\,dx + \int_{\Rnkp} |x|^{2}  \pxi w(x)^{2} \,dx\right. \\
        &\left.\qquad \qquad  \qquad  \qquad+ \int_{\Rnkp}   \pxi w(x)^{2}\,dx\right)\\
        &= \frac{2}{n+2k+3} \inf_{c \in \mathbb{R}} \left(
\int_{{\Rnkp}} \left| \nabla \left(u(x) - c\prod^n_{i=n-k+1}x_ie^{-\frac12 |x|^2} \right) \right|^2  dx   \right.\\ 
&\left. +\int_{{\Rnkp}} |x|^2 \left| u(x) - c \prod^n_{i=n-k+1}x_ie^{-\frac12 |x|^2} \right|^2 dx + \int_{{\Rnkp}} \left| u(x) - c \prod^n_{i=n-k+1}x_i e^{-\frac12 |x|^2} \right|^2 dx \right).
       \end{align*}
\end{proof}
Now we will look at stability of stability of the HUP. In \cite{LLR25} it has been established that for $u \in C^\infty_c(\RnAp)$,
\[
\begin{aligned}
&\Biggl(\int_{\RnAp} |\nabla u(x)|^2\,x^A dx\Biggr)^{\!1/2}
\Biggl(\int_{\RnAp} |u(x)|^2 |x|^2\,x^A dx\Biggr)^{\!1/2}
- \frac{n+\left\vert A\right\vert}{2}\int_{\RnAp} |u(x)|^2\,x^A dx \\
&\quad - \inf_{c\in\R,\,\lambda \neq 0}
   \int_{\RnAp} \Bigl|u(x) - c\,e^{-\tfrac{|x|^2}{2\lambda^2}}\Bigr|^2 x^A dx \\
&\;\;\ge \;\; \tfrac12 \inf_{c\in\R,\,\mathbf d \in \R^N,\,\lambda \neq 0}
   \left(\int_{\RnAp} \Bigl| u(x) - (c+\mathbf d \cdot x)\,e^{-\tfrac{|x|^2}{2\lambda^2}} \Bigr|^2 x^A dx\right)^{1/2}.
\end{aligned}
\]

This provides an improved scale-dependent stability estimate for the Heisenberg Uncertainty Principle with monomial weight, showing that near-extremizers must be quantitatively close to the Gaussian–affine extremal family.

Using this, we get the following result for stabitity of stability of HUP on orthant 

\begin{theorem}
   For $ u \in X(\Rnkp)$, 
\begin{align*}
&\Biggl(\int_{\Rnkp} |\nabla  u(x)|^2\, dx\Biggr)^{\!1/2}
\Biggl(\int_{\Rnkp} | u(x)|^2 |x|^2\, dx\Biggr)^{\!1/2}
- \frac{n+2k}{2}\int_{\Rnkp} | u(x)|^2\, dx \\
&\quad -\inf_{c\in \R}\int_{\Rnkp} \left| u(x) -ce^{-\frac{1}{2\alpha^2}|x|^2}\pxi  \right|^2 dx\\ 
& \hskip 1in \geq  \inf_{c\in \R.,\mathbf{d}\in\R^{n-k}\times {0}^k}\int_{\Rnkp} \left| u(x) -(c+\mathbf{d}\!\cdot\!x)e^{-\frac{1}{2\alpha^2}|x|^2}\pxi \right|^2 dx.
    \end{align*}
\end{theorem}
\begin{proof}
For $ u \in C^\infty_c(\Rnkp)$, by \eqref{eq:hupi} and Theorem \ref{T:weighted-poincare-stability}, we have
\begin{align*}
     &\Big(\int_{{\Rnkp}}|\nabla u|^2 \, dx\Big)^\frac{1}{2} \;  \Big(\int_{{\Rnkp}} |x|^2 \left\vert u(x) \right\vert^2 dx\Big)^\frac{1}{2}- \frac{(n + 2k)}{2} \left( \int_{{\Rnkp}} \left\vert u(x) \right\vert^2 \, dx\right)\\  
  &= \frac{\alpha^2}{2}\int_{\Rnkp}  | \nabla \Big( \frac{u(x)}{\prod_{i=n-k+1}^n x_i} e^{\frac{|x|^2}{2\alpha^2}} \Big) \Big|^2 e^{-\frac{1}{\alpha^2}|x|^2}\prod_{i=n-k+1}^n x_i^2 dx\\
&\geq \inf_{c\in \R}\int_{\Rnkp} \left|\Big(\frac{u(x)}{\prod_{i=n-k+1}^n x_i} e^{\frac{|x|^2}{2\alpha^2}} \Big)-c\  \right|^2 e^{-\frac{1}{\alpha^2}|x|^2}\prod_{i=n-k+1}^n x_i^2 dx\\ 
& \hskip 1in + \inf_{c\in \R,\mathbf{d}\in\R^{n-k}\times {0}^k}\int_{\Rnkp} \left| \Big( \frac{u(x)}{\prod_{i=n-k+1}^n x_i} e^{\frac{|x|^2}{2\alpha^2}} \Big)-c-\mathbf{d}\!\cdot\!x \right|^2 e^{-\frac{1}{\alpha^2}|x|^2}\prod_{i=n-k+1}^n x_i^2 dx\\
&\geq\inf_{c\in \R}\int_{\Rnkp} \left| u(x) -ce^{-\frac{1}{2\alpha^2}|x|^2}\pxi  \right|^2 dx\\ 
& \hskip 1in + \inf_{c\in \R.,\mathbf{d}\in\R^{n-k}\times {0}^k}\int_{\Rnkp} \left| u(x) -(c+\mathbf{d}\!\cdot\!x)e^{-\frac{1}{2\alpha^2}|x|^2}\pxi \right|^2 dx.
\end{align*}

\end{proof}


\begin{thebibliography}{999}

\bibitem{BGL14} D. Bakry, I. Gentil, M. Ledoux, Analysis and Geometry of Markov Diffusion Operators, Grundlehren Math. Wiss., 348, Springer, Cham, 2014, xx+552 pp.

\bibitem{BEL} A. A. Balinsky, W. D. Evans, R. T. Lewis, The Analysis and Geometry of Hardy's Inequality, Universitext, Springer, Cham, 2015, xv+263 pp.

\bibitem{BFT04} G. Barbatis, S. Filippas, A. Tertikas, A unified approach to improved $L^p$ Hardy inequalities with best constants, Trans. Amer. Math. Soc. 356 (2004), no. 6, 2169--2196.

\bibitem{BWW03} T. Bartsch, T. Weth, M. Willem, A Sobolev inequality with remainder term and critical equations on domains with topology for the polyharmonic operator, Calc. Var. Partial Differential Equations 18 (2003), 253--268.

\bibitem{Beckner93} W. Beckner, Sharp Sobolev inequalities on the sphere and the Moser-Trudinger inequality, Ann. of Math. (2) 138 (1993), no. 1, 213--242.

\bibitem{BE91} G. Bianchi, H. Egnell, A note on the Sobolev inequality, J. Funct. Anal. 100 (1991), 18--24.

\bibitem{BDNN20} M. Bonforte, J. Dolbeault, B. Nazaret, N. Nikita, Stability in Gagliardo-Nirenberg-Sobolev inequalities, Flows, Regularity and the Entropy Method, Mem. Amer. Math. Soc. 308 (2025), no. 1554, viii+166 pp..

\bibitem{BL85} H. Brezis, E. H. Lieb, Sobolev inequalities with remainder terms, J. Funct. Anal. 62 (1985), 73--86.

\bibitem{BM98} H. Brezis, M. Marcus, Hardy's inequalities revisited, Ann. Scuola Norm. Sup. Pisa Cl. Sci. (4) 25 (1997), no. 1--2, 217--237 (1998).

\bibitem{BV} H. Brezis, J. L. V\'azquez, Blow-up solutions of some nonlinear elliptic problems, Rev. Mat. Univ. Complut. Madrid 10 (1997), 443--469.

\bibitem{CF13} E. Carlen, A. Figalli, Stability for a GNS inequality and the log-HLS inequality, with application to the critical mass Keller-Segel equation, Duke Math. J. 162 (2013), no. 3, 579--625.

\bibitem{Caz10} C. Cazacu, On Hardy inequalities with singularities on the boundary, C. R. Math. Acad. Sci. Paris 349 (2011), no. 5--6, 273--277.

\bibitem{CFLL24} C. Cazacu, J. Flynn, N. Lam, G. Lu, Caffarelli-Kohn-Nirenberg identities, inequalities and their stabilities, J. Math. Pures Appl. (9) 182 (2024), 253--284.

\bibitem{CKLL24} C. Cazacu, D. Krej\v{c}i\v{r}\'{\i}k, N. Lam, A. Laptev, Hardy inequalities for magnetic $p$-Laplacians, Nonlinearity 37 (2024), no. 3, Paper No. 035004, 27 pp.

\bibitem{CZ13} C. Cazacu, E. Zuazua, Improved multipolar Hardy inequalities, in: Studies in Phase Space Analysis with Applications to PDEs, 35--52, Progr. Nonlinear Differential Equations Appl., 84, Birkh\"auser/Springer, New York, 2013.

\bibitem{CGMSO18} H. Chan, N. Ghoussoub, S. Mazumdar, S. Shakerian, L. F. de Oliveira Faria, Mass and extremals associated with the Hardy-Schr\"odinger operator on hyperbolic space, Adv. Nonlinear Stud. 18 (2018), no. 4, 671--689.

\bibitem{CLT23} L. Chen, G. Lu, H. Tang, Sharp stability of log-Sobolev and Moser-Onofri inequalities on the sphere, J. Funct. Anal. 285 (2023), no. 5, Paper No. 110022, 24 pp.

\bibitem{CLT24} L. Chen, G. Lu, H. Tang, Stability of Hardy-Littlewood-Sobolev inequalities with explicit lower bounds, Adv. Math. 450 (2024), Paper No. 109778, 28 pp.

\bibitem{CLT242} L. Chen, G. Lu, H. Tang, Optimal asymptotic lower bound for stability of fractional Sobolev inequality and the global stability of log-Sobolev inequality on the sphere, Adv. Math. 479 (2025), part B, Paper No. 110438, 32 pp.

\bibitem{CLT243} L. Chen, G. Lu, H. Tang, Optimal stability of Hardy-Littlewood-Sobolev and Sobolev inequalities of arbitrary orders with dimension-dependent constants, arXiv:2405.17727, to appear in Math. Ann.

\bibitem{CLTW} L. Chen, G. Lu, H. Tang, B. Wang, Asymptotically sharp stability of Sobolev inequalities on the Heisenberg group with dimension-dependent constants, J. Math. Pures Appl. (9) 206 (2026), Paper No. 103832, 29 pp.

\bibitem{CFW13} S. Chen, R. Frank, T. Weth, Remainder terms in the fractional Sobolev inequality, Indiana Univ. Math. J. 62 (2013), no. 4, 1381--1397.

\bibitem{CFMP09} A. Cianchi, N. Fusco, F. Maggi, A. Pratelli, The sharp Sobolev inequality in quantitative form, J. Eur. Math. Soc. (JEMS) 11 (2009), no. 5, 1105--1139.

\bibitem{DD14} L. D'Ambrosio, S. Dipierro, Hardy inequalities on Riemannian manifolds and applications, Ann. Inst. H. Poincar\'e Anal. Non Lin\'eaire 31 (2014), no. 3, 449--475.

\bibitem{DDLL25} N. A. Dao, A. X. Do, N. Lam, G. Lu, Sobolev interpolation inequalities with optimal Hardy-Rellich inequalities and critical exponents, Calc. Var. Partial Differential Equations 64 (2025), no. 7, Paper No. 204.

\bibitem{DoFLL23} A. X. Do, J. Flynn, N. Lam, G. Lu, $L^p$-Caffarelli-Kohn-Nirenberg inequalities and their stabilities, arXiv:2310.07083.

\bibitem{DoLL24} A. X. Do, N. Lam, G. Lu, A new approach to weighted Hardy-Rellich inequalities: improvements, symmetrization principle and symmetry breaking, J. Geom. Anal. 34 (2024), no. 12, Paper No. 363, 28 pp.

\bibitem{DLLZ} A. Do, N. Lam, G. Lu, Sharp stability of the Heisenberg Uncertainty Principle: Second-Order and Curl-Free Field Cases, J. Funct. Anal. 290 (2026), no. 7, Paper No. 111321. 

\bibitem{DEFFL} J. Dolbeault, M. J. Esteban, A. Figalli, R. Frank, M. Loss, Sharp stability for Sobolev and log-Sobolev inequalities, with optimal dimensional dependence, Camb. J. Math. 13 (2025), no. 2, 359–430.

\bibitem{DLL22} N. T. Duy, N. Lam, G. Lu, $p$-Bessel pairs, Hardy's identities and inequalities and Hardy-Sobolev inequalities with monomial weights, J. Geom. Anal. 32 (2022), no. 4, Paper No. 109, 36 pp.

\bibitem{DPH25} N. T. Duy, N. V. Phong, P. T. T. Hien, Hardy inequalities with Bessel pair for Dunkl operator, Adv. Nonlinear Stud. 25 (2025), no. 4, 1127--1141.

\bibitem{Fall12} M. M. Fall, On the Hardy-Poincaré inequality with boundary singularities, Commun. Contemp. Math. 14 (2012), no. 3, 1250019, 13 pp.

\bibitem{FM12} M. M. Fall, R. Musina, Hardy-Poincaré inequalities with boundary singularities, Proc. Roy. Soc. Edinburgh Sect. A 142 (2012), no. 4, 769--786.

\bibitem{F21} M. Fathi, A short proof of quantitative stability for the Heisenberg-Pauli-Weyl inequality, Nonlinear Anal. 210 (2021), Paper No. 112403, 3 pp.

\bibitem{FJ1} A. Figalli, D. Jerison, Quantitative stability for sumsets in $\mathbb{R}^n$, J. Eur. Math. Soc. (JEMS) 17 (2015), no. 5, 1079--1106.

\bibitem{FJ2} A. Figalli, D. Jerison, Quantitative stability for the Brunn-Minkowski inequality, Adv. Math. 314 (2017), 1--47.

\bibitem{FN19} A. Figalli, R. Neumayer, Gradient stability for the Sobolev inequality: the case $p \geq 2$, J. Eur. Math. Soc. (JEMS) 21 (2019), no. 2, 319--354.

\bibitem{FZ22} A. Figalli, Y. Zhang, Sharp gradient stability for the Sobolev inequality, Duke Math. J. 171 (2022), no. 12, 2407--2459.

\bibitem{FTT09} S. Filippas, A. Tertikas, J. Tidblom, On the structure of Hardy–Sobolev–Maz'ya inequalities, J. Eur. Math. Soc. (JEMS) 11 (2009), no. 6, 1165--1185.

\bibitem{Flynn20} J. Flynn, Sharp Caffarelli-Kohn-Nirenberg-type inequalities on Carnot groups, Adv. Nonlinear Stud. 20 (2020), no. 1, 95--111.

\bibitem{FLL22} J. Flynn, N. Lam, G. Lu, Hardy-Poincaré-Sobolev type inequalities on hyperbolic spaces and related Riemannian manifolds, J. Funct. Anal. 283 (2022), no. 12, Paper No. 109714, 37 pp.

\bibitem{FLL21} J. Flynn, N. Lam, G. Lu, Sharp Hardy identities and inequalities on Carnot groups, Adv. Nonlinear Stud. 21 (2021), no. 2, 281--302.

\bibitem{FLL25} J. Flynn, N. Lam, G. Lu, $L^p$-Hardy identities and inequalities with respect to the distance and mean distance to the boundary, Calc. Var. Partial Differential Equations 64 (2025), no. 1, Paper No. 22, 39 pp.

\bibitem{FLLM23} J. Flynn, N. Lam, G. Lu, S. Mazumdar, Hardy's identities and inequalities on Cartan-Hadamard manifolds, J. Geom. Anal. 33 (2023), no. 1, Paper No. 27, 34 pp.

\bibitem{FLY24} J. Flynn, G. Lu, Q. Yang, Sharp Hardy-Sobolev-Maz'ya, Adams and Hardy-Adams inequalities on quaternionic hyperbolic spaces and on the Cayley hyperbolic plane, Rev. Mat. Iberoam. 40 (2024), no. 2, 403--462.

\bibitem{Fol97} G. B. Folland, A. Sitaram, The uncertainty principle: a mathematical survey, J. Fourier Anal. Appl. 3 (1997), no. 3, 207--238.

\bibitem{FS08} R. Frank, R. Seiringer, Non-linear ground state representations and sharp Hardy inequalities, J. Funct. Anal. 255 (2008), no. 12, 3407--3430.

\bibitem{GR23} D. Ganguly, P. Roychowdhury, Improved Poincaré-Hardy inequalities on certain subspaces of the Sobolev space, Proc. Amer. Math. Soc. 151 (2023), no. 8, 3513--3527.

\bibitem{GL2017} F. Gesztesy, L. L. Littlejohn, Factorizations and Hardy-Rellich-type inequalities, Nonlinear Partial Differential Equations, Mathematical Physics, and Stochastic Analysis, 207--226, EMS Ser. Congr. Rep., Eur. Math. Soc., Zürich, 2018.

\bibitem{GLMP} F. Gesztesy, L. L. Littlejohn, I. Michael, M. Pang, Radial and logarithmic refinements of Hardy's inequality, St. Petersburg Math. J. 30 (2019), no. 3, 429--436; Algebra i Analiz 30 (2018), no. 3, 55--65.

\bibitem{GLMW} F. Gesztesy, L. L. Littlejohn, I. Michael, R. Wellman, On Birman's sequence of Hardy-Rellich-type inequalities, J. Differential Equations 264 (2018), no. 4, 2761--2801.

\bibitem{GM1} N. Ghoussoub, A. Moradifam, Functional Inequalities: New Perspectives and New Applications, Mathematical Surveys and Monographs, 187. Amer. Math. Soc., Providence, RI, 2013, xxiv+299 pp.

\bibitem{Gro75} L. Gross, Logarithmic Sobolev inequalities, Amer. J. Math. 97 (1975), no. 4, 1061--1083.

\bibitem{HT25} N. Hamamoto, F. Takahashi, A curl-free improvement of the Rellich–Hardy inequality with weight, Adv. Nonlinear Stud. 25 (2025), no. 4, 1204–1234.

\bibitem{Heis1927} W. Heisenberg, Über den anschaulichen Inhalt der quantentheoretischen Kinematik und Mechanik, Z. Physik 43 (1927), 172--198.

\bibitem{HY22} X. Huang, D. Ye, First order Hardy inequalities revisited, Commun. Math. Res. 38 (2022), no. 4, 535--559.

\bibitem{Kennard1927} E. H. Kennard, Zur Quantenmechanik einfacher Bewegungstypen, Z. Physik 44 (1927), 326--352.

\bibitem{KMP2007} A. Kufner, L. Maligranda, L.-E. Persson, The Hardy Inequality: About its History and Some Related Results, Vydavatelský Servis, Pilsen, 2007.

\bibitem{KP} A. Kufner, L.-E. Persson, N. Samko, Weighted Inequalities of Hardy Type, 2nd ed., World Scientific, Hackensack, NJ, 2017, xx+459 pp.

 

 

\bibitem{LL23} N. Lam, G. Lu, Improved $L^p$-Hardy and $L^p$-Rellich inequalities with magnetic fields, Vietnam J. Math. 51 (2023), no. 4, 971--984.

\bibitem{LLR25} N. Lam, G. Lu, A. Russanov, Stability of Gaussian Poincaré inequalities and the Heisenberg uncertainty principle with monomial weights, Math. Z. 312 (2026), no. 2, Paper No. 42.

\bibitem{LLZ19} N. Lam, G. Lu, L. Zhang, Factorizations and Hardy's type identities and inequalities on upper half spaces, Calc. Var. Partial Differential Equations 58 (2019), no. 6, Paper No. 183, 31 pp.

\bibitem{LLZ20} N. Lam, G. Lu, L. Zhang, Geometric Hardy's inequalities with general distance functions, J. Funct. Anal. 279 (2020), no. 8, 108673, 35 pp.

\bibitem{LuYang2} G. Lu, Q. Yang, Green's functions of Paneitz and GJMS operators on hyperbolic spaces and sharp Hardy-Sobolev-Maz'ya inequalities on half spaces, Adv. Math. 398 (2022), Paper No. 108156, 42 pp.

\bibitem{LuYang1} G. Lu, Q. Yang, Paneitz operators and Hardy-Sobolev-Maz'ya inequalities for higher order derivatives on half spaces, Amer. J. Math. 141 (2019), no. 6, 1777--1816.

\bibitem{LuYang3} G. Lu, Q. Yang, Sharp Hardy-Sobolev-Maz'ya, Adams and Hardy-Adams inequalities on the Siegel domains and complex hyperbolic spaces, Adv. Math. 405 (2022), Paper No. 108512, 62 pp.

\bibitem{LW99} G. Lu, J. Wei, On a Sobolev inequality with remainder terms, Proc. Amer. Math. Soc. 128 (1999), 75--84.

\bibitem{MV21} S. McCurdy, R. Venkatraman, Quantitative stability for the Heisenberg-Pauli-Weyl inequality, Nonlinear Anal. 202 (2021), Paper No. 112147, 13 pp.

\bibitem{NVH20} V. H. Nguyen, New sharp Hardy and Rellich type inequalities on Cartan--Hadamard manifolds and their improvements, Proc. Roy. Soc. Edinburgh Sect. A 150 (2020), no. 6, 2952--2981.

\bibitem{NVH19} V. H. Nguyen, Weighted Finsler trace Hardy inequality on half spaces, J. Math. Anal. Appl. 474 (2019), no. 2, 1198--1212.

\bibitem{OK} B. Opic, A. Kufner, Hardy-type inequalities, Pitman Research Notes in Mathematics Series, 219. Longman Scientific \& Technical, Harlow, 1990. xii+333 pp.


\bibitem{SuYang2012} D. Su, Q. Yang, On the best constants of Hardy inequality in $\mathbb{R}^{n-k} \times (\mathbb{R}_{+})^k$ and related improvements, J. Math. Anal. Appl. 389 (2012), 48--53.

\bibitem{Wang22} J. Wang, $L^p$ Hardy's identities and inequalities for Dunkl operators, Adv. Nonlinear Stud. 22 (2022), no. 1, 416--435.

\bibitem{YSK15} Q. Yang, D. Su, Y. Kong, Improved Hardy inequalities for Grushin operators, J. Math. Anal. Appl. 424 (2015), no. 1, 321--343.

\end{thebibliography}
\end{document}